\documentclass[final,12pt]{elsarticle}
\usepackage{amsmath}
\usepackage{amssymb}
\usepackage{hyperref}

\def\diag{\mathop{\rm diag}\nolimits}

\def\ker{\mathop{\rm Ker}\nolimits}

\def\rank{\mathop{\rm rank}\nolimits}
\def\nrank{\mathop{\rm nrank}\nolimits}

\newcommand{\CC}{\mathbb C}

\newcommand{\eps}{\varepsilon}

\newtheorem{theorem}{Theorem}[section]
\newtheorem{remark}[theorem]{Remark}
\newtheorem{definition}[theorem]{Definition}
\newtheorem{proposition}[theorem]{Proposition}

\newtheorem{example}[theorem]{Example}
\newtheorem{lemma}[theorem]{Lemma}

\newproof{proof}{{\bf Proof}}

\newtheorem{athm}{Theorem}[section]
\newtheorem{alem}[athm]{Lemma}

\newcommand{\smtxa}[2]{
{\mbox{\scriptsize
$\left[\!\! \begin{array}{#1} #2 \end{array} \!\! \right]$}}}
\newcommand{\mtxa}[2]{\left[\!\!\begin{array}{#1} #2 \end{array}\!\!\right]}

\newcommand{\ph}{\phantom}

\newcommand{\ch}[1]{#1} % switch blue (changed) part to black
\newcommand{\cm}[1]{#1} 

\title{Numerical methods for eigenvalues \\
of singular polynomial eigenvalue problems}

\journal{Linear Algebra Appl.}

\begin{document}
\begin{frontmatter}

\author[au1]{Michiel E.~Hochstenbach}
\ead{m.e.hochstenbach@tue.nl}
\cortext[cor1]{Version \today. Corresponding author}
\address[au1]{Department of Mathematics and Computer Science, TU Eindhoven,
PO Box 513, 5600 MB, The Netherlands.}
\author[au2]{Christian Mehl\corref{cor1}}
\ead{mehl@math.tu-berlin.de}
\address[au2]{Institut f\"{u}r Mathematik, Technische Universit\"at Berlin, Sekretariat MA 4-5,
Stra\ss e des 17.~Juni 136, 10623 Berlin, Germany.}

\author[au3]{Bor~Plestenjak}
\ead{bor.plestenjak@fmf.uni-lj.si}
\address[au3]{IMFM and Faculty of Mathematics and Physics, University of Ljubljana,
Jadranska 19, SI-1000 Ljubljana, Slovenia.}

\begin{abstract}
Recently, three numerical methods for the computation of
eigenvalues of singular matrix pencils, based on a rank-completing perturbation,
a rank-projection, or an augmentation have been developed.
We show that all three approaches can be
generalized to treat singular polynomial eigenvalue problems. The common denominator of all three approaches is a
transformation of a singular into a regular matrix polynomial whose eigenvalues
are a disjoint union of the eigenvalues of the singular polynomial, called true
eigenvalues, and additional fake eigenvalues. The true eigenvalues can then be
separated from the fake eigenvalues using information on the corresponding
left and right eigenvectors.
We illustrate the approaches on several interesting applications, including bivariate polynomial systems and ZGV points.
\end{abstract}


\end{frontmatter}

\section{Introduction}\label{sec:introduction}
We consider the numerical computation of finite eigenvalues of the \emph{singular polynomial eigenvalue problem} (singular PEP) associated with a matrix polynomial
\begin{equation}\label{eq:P}
P(\lambda)=A_0+\lambda A_1 + \cdots + \lambda^d A_d
\end{equation}
of degree $d\ge 2$, where $A_0, \dotsc,A_d$, $A_d\ne 0$, are $m\times n$ (real or complex) matrices such that the polynomial $P(\lambda)$ is \emph{singular}, which means that either $m=n$ and $\det\big(P(\lambda)\big)\equiv 0$, or $m\ne n$.
Then $\lambda_0\in\CC$ is
a \emph{finite eigenvalue} of $P$ if $\rank\big(P(\lambda_0)\big)<\nrank(P)$, where
\[
\nrank(P) := \texorpdfstring{\max_{\zeta\in\CC}\,\rank\big(P(\zeta)\big)}{max rank of P}
\]
is the \emph{normal rank} of the matrix polynomial $P$. (See Definition~\ref{def:rev} for infinite eigenvalues.)
The case $d=1$, where \eqref{eq:P} is a singular matrix pencil, is covered
in \cite{HMP_SingGep} and \cite{HMP_SingGep2}. In this paper we show how we can extend the main ideas and tools to singular PEPs.

A standard approach to solve a PEP is to first linearize the matrix polynomial
into a linear matrix pencil and then compute the eigenvalues (and possibly eigenvectors) from the obtained generalized eigenvalue problem. If we apply a strong linearization
(see Section~\ref{sec:auxiliary} for details) to a singular matrix polynomial $P$,
we obtain a singular matrix pencil $L$ such that
its finite and infinite eigenvalues agree with the eigenvalues of $P$. To compute the eigenvalues of $L$ we may, e.g., apply a numerical method from \cite{HMP_SingGep} or \cite{HMP_SingGep2}, where we perturb or project $L$ into a regular matrix pencil $\widetilde L$ such that all eigenvalues of $L$ are also
eigenvalues of $\widetilde L$. Then the eigenvalues of $P$ can be extracted from the eigenvalues of $\widetilde L$ using orthogonality relations of the corresponding left
and right eigenvectors.

In this paper we suggest an alternative approach. Instead of applying a linearization,
we first perturb, project, or augment a singular matrix polynomial $P$ into a regular
matrix polynomial $\widetilde P$ such that all eigenvalues of $P$ are also eigenvalues of $\widetilde P$, \ch{which has the advantage that the polynomial structure of the problem is preserved}.
Then we compute the eigenvalues of $\widetilde P$ together with the left and right eigenvectors and extract
the finite eigenvalues of $P$ from this set. To compute the eigenvalues of the regular matrix polynomial $\widetilde P$ we may apply any of the many
numerical methods for such problems, including of course the use of linearizations and
the solution of the corresponding generalized eigenvalue problem.
In that case, one may wonder whether it is equivalent to first linearize $P$ into $L$
and then perturb (project or augment) $L$ into a regular pencil $\widetilde L$, or first
perturb (project or augment) $P$ into a regular matrix polynomial $\widetilde P$ and then
linearize $\widetilde P$ as a pencil $\widehat L$. We will show that these methods are
not equivalent and that the second approach has the advantage of leading to a
generalized eigenvalue problem of smaller size.

A different numerical method for singular quadratic eigenvalue problems (singular QEPs) has recently been proposed in \cite{DanielIvana}. In this method a singular QEP with $Q(\lambda)=\lambda^2 M + \lambda C + K$ is perturbed
into a regular problem with $\widetilde Q(\lambda)=\lambda^2
\widetilde M + \lambda \widetilde C + \widetilde K$ by small
random full rank perturbations. Next, the eigenvalues $\lambda_i$ of $\widetilde Q$ are
computed together with the right and left eigenvectors %$x_i$ and $y_i$.
and if the condition number of a computed eigenvalue is small enough, then the eigenvalue is
identified as a finite eigenvalue of $Q$. The corresponding criterion is based on the
notion of the \emph{$\delta$-weak condition number}
of an eigenvalue of a singular PEP, as introduced
in \cite{LotzNoferini}.

The alternative approach proposed in this paper has the advantage that it leaves the
eigenvalues intact and is therefore expected to return more accurate solutions.
If we apply the projection variant, then another advantage are smaller matrices that
require less computational work. A drawback of the new methods is that they rely on a correct determination of the normal rank.
We give a more detailed comparison in Section~\ref{sec:num_results}.

The rest of the paper is organized as follows.
Section~\ref{sec:auxiliary} introduces the background theory on singular matrix polynomials. In
Section~\ref{sec:main} we generalize the main results from \cite{HMP_SingGep} and \cite{HMP_SingGep2} for
singular pencils to the setting of singular matrix polynomials. In
Section~\ref{sec:method} we then present the corresponding numerical methods that are
directly applied to the given matrix polynomial and compare them in
Section~\ref{sec:compare_lin} to the approach that first applies a linearization.
We illustrate the theoretical results and algorithms with numerical experiments in Section~\ref{sec:num_results}, where we also describe applications to bivariate
polynomial systems and the computation of ZGV points \ch{that are defined in Section~\ref{subs:ZGV}}. Finally,
we summarize some conclusions in Section~\ref{sec:conclusion}.

\section{Preliminaries}\label{sec:auxiliary}
For the following theory on eigenvalues, nullspaces, and minimal indices of singular matrix polynomials we refer to,
e.g., \cite{DeTeranDopicoMackey09} or \cite{DeTeranDopicoMackey_IndexSumTheorem}. First, to define
multiplicities of eigenvalues of a singular matrix polynomial, we recall the Smith form which is obtained under unimodular equivalence. We remind the reader that a
square matrix polynomial is called \emph{unimodular} if its determinant is a
nonzero constant which is equivalent to saying that it is invertible when interpreted
as a matrix over the field $\mathbb C(\lambda)$ of rational functions over $\mathbb C$.

\begin{theorem}[Smith form] \label{the:smith}
Let $P(\lambda)$ be an $m\times n$ matrix polynomial of normal rank $r$. Then there exist
unimodular matrix polynomials $E(\lambda)$ and $F(\lambda)$ of sizes $m\times m$ and $n\times n$, respectively,
such that
\begin{equation}\label{eq:SmithForm}
E(\lambda)\,P(\lambda)\,F(\lambda) = \diag\big(d_1(\lambda), \dotsc, d_r(\lambda), 0, \dotsc, 0\big) =: D(\lambda),
\end{equation}
where $d_1(\lambda), \dotsc, d_r(\lambda)$ are monic polynomials such
that
$d_j(\lambda)$ is a divisor of $d_{j+1}(\lambda)$ for $j=1, \dotsc,r-1$. The $m\times n$ diagonal matrix polynomial $D(\lambda)$ is unique.
\end{theorem}

The nonzero diagonal elements $d_1(\lambda), \dotsc, d_r(\lambda)$ in the Smith form \eqref{eq:SmithForm} are
called the \emph{invariant polynomials} of $P(\lambda)$.

\begin{definition}
Let $\lambda_0\in\CC$ be a finite eigenvalue of a matrix polynomial
$P(\lambda)$ of normal rank $r$. The invariant polynomials $d_1(\lambda), \dotsc, d_r(\lambda)$ of $P(\lambda)$ can then be
uniquely factored as
\[
d_i(\lambda)=(\lambda-\lambda_0)^{\alpha_i}\,p_i(\lambda),\ \textrm{where}\ \alpha_i\ge 0\
\ \textrm{and}\ \ p_i(\lambda_0)\ne 0,
\]
for $i=1, \dotsc,r$. The exponents $0\le \alpha_1\le \cdots \le \alpha_r$ are called the
\emph{structural indices} of $P(\lambda)$ at $\lambda_0$. The
\emph{algebraic multiplicity} of $\lambda_0$ is the sum $\alpha_1+\cdots+\alpha_r$ of the structural indices, while
the \emph{geometric multiplicity} is the number of nonzero structural indices.
\end{definition}

\begin{definition} \label{def:rev}
Let $P(\lambda)$ be a matrix polynomial \eqref{eq:P} of degree $d$. Then
\[
{\rm rev}\,P(\lambda):=\lambda^j\,P(1/\lambda)=A_d+\lambda A_{d-1} + \cdots +\lambda^d A_0
\]
is the \emph{reversal} of $P(\lambda)$. We say that $\lambda_0=\infty$ is an eigenvalue of
$P(\lambda)$ if $0$ is an eigenvalue of ${\rm rev}\,P(\lambda)$. The structural
indices of $P(\lambda)$ at $\lambda_0=\infty$ and the algebraic and geometric multiplicity
of eigenvalue $\lambda_0=\infty$ are then defined as structural indices, and the algebraic and geometric multiplicity of ${\rm rev}\,P(\lambda)$ at $0$,
respectively.
\end{definition}

\begin{definition} The \emph{right} and \emph{left nullspaces}
of an $m\times n$
matrix polynomial $P(\lambda)$
are the vector spaces
\begin{align*}
{\cal N}_r(P)&=\{\,x(\lambda)\in\CC(\lambda)^n: P(\lambda)\,x(\lambda)\equiv 0\,\},\cr
{\cal N}_l(P)&=\{\,y(\lambda)\in\CC(\lambda)^m: y(\lambda)^*P\,(\lambda)\equiv 0\,\}
\end{align*}
of rational vectors
$x(\lambda)$ and $y(\lambda)$ annihilated by $P(\lambda)$.
\end{definition}

Clearly, the following identity holds:
\[
\nrank(P)=n-\dim {\cal N}_r(P)=m-\dim {\cal N}_l(P).
\]
For square matrix polynomials, i.e., when $m=n$, it then follows that $\dim {\cal N}_r(P)=\dim {\cal N}_l(P)$.

A basis of a subspace ${\cal V}$ of $\CC(\lambda)^n$ is called a \emph{polynomial basis} \cite{Forney_minimal_basis}. A polynomial basis is called \emph{minimal} if
the sum of the degrees of its polynomials is minimal along all
polynomial bases of ${\cal V}$.

\begin{definition} Let $P(\lambda)$ be an $m\times n$ singular matrix polynomial of normal rank $r$. Let
$\{x_1(\lambda), \dotsc, x_{n-r}(\lambda)\}$
and
$\{y_1(\lambda), \dotsc, y_{m-r}(\lambda)\}$
be minimal bases of, respectively, the right and left nullspaces of
$P(\lambda)$, ordered so that $\eps_1\le \cdots \le \eps_{n-r}$ and $\eta_1\le \cdots \le \eta_{m-r}$, where
$\eps_i=\deg(x_i)$ for $i=1, \dotsc,n-r$ and
$\eta_i=\deg(y_i)$ for $i=1, \dotsc,m-r$.
Then $\eps_1,\dots,\eps_{n-r}$ and
$\eta_1,\dots,\eta_{m-r}$ are, respectively, the
\emph{right} and \emph{left minimal} indices of $P(\lambda)$.
\end{definition}

The minimal bases are not unique, but the minimal indices are. If
$\{x_1(\lambda), \dotsc, x_{n-r}(\lambda)\}$ and
$\{y_1(\lambda), \dotsc, y_{m-r}(\lambda)\}$ are minimal bases for the right and left nullspaces of
$P(\lambda)$, then the \emph{right} and \emph{left singular spaces at}
$\mu\in\mathbb C$ are defined as \
\begin{align*}
\ker_\mu\big(P(\lambda)\big)&:={\rm span}\big(x_1(\mu), \dotsc, x_{n-r}(\mu)\big), \\
\ker_\mu\big(P(\lambda)^*\big)&:={\rm span}\big(y_1(\mu), \dotsc, y_{m-r}(\mu)\big).
\end{align*}
The right and left singular spaces at $\mu$ do not depend on the choice of the minimal bases \cite[Lemma 2.8]{DopicoNoferini_RootPol}.

The \emph{index sum theorem} from \cite[Lemma~6.3]{DeTeranDopicoMackey_IndexSumTheorem} will play an important role
in the next section. The result can be used to compute the number of
eigenvalues of a matrix polynomial if the minimal indices are known.

\begin{theorem}[Index sum theorem]\label{thm:indexsum}
Let $P(\lambda)$ be an $m\times n$ singular matrix polynomial of degree $d$ and normal rank $r$ with
right minimal indices
$\eps_1\le \cdots \le \eps_{n-r}$ and
left minimal indices
$\eta_1\le \cdots \le \eta_{m-r}$. Then
\begin{equation}\label{eq:IndexSum}
d\,r = \delta_{\rm fin}(P)+\delta_{\infty}(P)+\sum_{i=1}^{n-r}\eps_i+\sum_{i=1}^{m-r}\eta_i,
\end{equation}
where
$\delta_{\rm fin}(P)$ is the sum of the algebraic multiplicities of all finite eigenvalues of
$P$ and $\delta_{\rm \infty}(P)$ is the algebraic multiplicity of eigenvalue
$\infty$.
\end{theorem}

If $\lambda_0\in\CC$ is an eigenvalue of an $m\times n$ singular matrix polynomial $P(\lambda)$ of degree $d$ and normal rank $r$, then $\dim\big(\ker(P(\lambda_0))\big)>n-r$ and as
a representative of a right eigenvector we can take
each
$x_0\in[x]\in\ker\big(P(\lambda_0)\big)\,/\,\ker_{\lambda_0}\big(P(\lambda)\big)$; see, e.g.,
\cite{DopicoNoferini_RootPol} or \cite{LotzNoferini}. In a similar way each
$y_0\in[y]\in\ker\big((P(\lambda_0)^*)\big)\,/\,\ker_{\lambda_0}\big(P(\lambda)^*\big)$ is a representative of a left eigenvector.
If $\lambda_0$ is geometrically simple, i.e., $\dim\big(\ker(P(\lambda_0))\big)=n-r+1$, then we
can uniquely (up to a scalar) define the right eigenvector and left eigenvector by the additional
requirement that $x$ and $y$ are respectively orthogonal to $\ker_{\lambda_0}\big(P(\lambda)\big)$ and
$\ker_{\lambda_0}\big(P(\lambda)^*\big)$.

The definition of the condition number for a \ch{simple} finite eigenvalue is based on
a restriction of the expansion result
from~\cite[Thm.~2 and Eq.~(18)]{DTFM_10_1stOrderPolynomials} to
the case of a simple eigenvalue; see also \cite[Thm.~3.2]{LotzNoferini}.

\begin{theorem}\label{thm:DeTeranDopico}
Let $\lambda_0$ be a finite simple eigenvalue of an $n\times n$ matrix polynomial $P(\lambda)$ of degree $d$ that has normal rank $r$. Let
$X=[X_1\ \, x]$ be an $n\times (n-r+1)$ matrix with orthonormal columns such that the columns of $X_1$ form a basis for $\ker_{\lambda_0}\big(P(\lambda)\big)$ and the columns
of $X$ form a basis for $\ker\big(P(\lambda_0)\big)$, and let
$Y=[Y_1\ \, y]$ be an $n\times (n-r+1)$ matrix with orthonormal columns such that the columns of $Y_1$ form a basis for $\ker_{\lambda_0}\big(P(\lambda)^*\big)$
and the columns of $Y$ form a basis for ${\ker}\big(P(\lambda_0)^*\big)$.
If $E(\lambda)$ is an $n\times n$ matrix polynomial of degree $d$
such that $Y_1^*E(\lambda_0)\,X_1$ is nonsingular,
then, for sufficiently small $\eps>0$,
there exists an eigenvalue $\lambda_0(\eps)$ of the perturbed pencil
$P(\lambda)+\eps \, E(\lambda)$ such that
\begin{equation}\label{eq:ocena_det_det}
 \lambda_0(\eps)=\lambda_0 - \frac{\det(Y^*E(\lambda_0)\,X)}
{(y^*P'(\lambda_0)\,x)\cdot\det(Y_1^*E(\lambda_0)\,X_1)}\,\eps +{\cal O}(\eps^2).
\end{equation}
\end{theorem}

In the regular case, when $n=r$, the above result simplifies into a well-known
expression for the eigenvalue expansion of a simple eigenvalue of a regular matrix polynomial \cite{Tisseur_BackwardErrorPEP}, namely
\begin{equation}\label{eq:ocena_det_det_regular}
 \lambda_0(\eps)=\lambda_0 - \frac{y^*E(\lambda_0)\,x}
 {y^*P'(\lambda_0)\,x}\,\eps +{\cal O}(\eps^2).
\end{equation}
Based on Theorem~\ref{thm:DeTeranDopico} the condition number of a simple eigenvalue can be generalized to singular matrix polynomials.

\begin{definition}\label{df:kappa}
Let $\lambda_0$ be a finite simple eigenvalue of an $n\times n$ singular matrix
polynomial $P(\lambda)$ of degree $d$ that has normal rank $r$. Let
$X=[X_1\ \, x]$ be an $n\times (n-r+1)$ matrix with orthonormal columns such that columns of $X_1$ form a basis for $\ker_{\lambda_0}\big(P(\lambda)\big)$ and columns of $X$ form a basis for
$\ker\big(P(\lambda_0)\big)$, and let
$Y=[Y_1\ \, y]$ be an $n\times (n-r+1)$ matrix with orthonormal columns such that columns of $Y_1$ form a basis for $\ker_{\lambda_0}\big(P(\lambda)^*\big)$
and columns of $Y$ form a basis for
${\ker}\big(P(\lambda_0)^*\big)$. Then we define
\begin{equation}\label{def:kp}
\gamma(\lambda_0)=|y^*P'(\lambda_0)\,x|\cdot \big(1+|\lambda_0|^2+\cdots+|\lambda_0|^{2d}\big)^{-1/2}
\end{equation}
and consider $\kappa(\lambda_0)=\gamma(\lambda_0)^{-1}$ as the \emph{condition number} of
$\lambda_0$.
\end{definition}

In the regular case $\kappa(\lambda_0)$ is indeed the condition number of
$\lambda_0$. Namely, if we perturb
$P(\lambda)=A_0+\lambda A_1+\cdots+\lambda^d A_d$ into
$\widetilde P(\lambda)=P(\lambda)+\eps E(\lambda)$, where
$E(\lambda) = E_0+\lambda E_1 +\cdots+\lambda^d E_d$ with
$\|E\|_*:=\left(\|E_0\|^2+\cdots+\|E_d\|^2\right)^{1/2} \le 1$, and
$\lambda_0(\eps)$ is the corresponding perturbed eigenvalue of $\widetilde P$, then
\[
|\lambda_0(\eps)-\lambda_0|\le \kappa(\lambda_0) \, \eps+{\cal O}(\eps^2)\,;
\]
see, e.g., \cite{Tisseur_BackwardErrorPEP} or \cite{DanielIvana}.

On the other hand, if $\lambda_0\in\CC$ is an eigenvalue of a singular matrix polynomial $P$, then arbitrarily small perturbations can move $\lambda_0$ to an arbitrary complex number as $\det(Y_1^*E(\lambda_0)X_1)$ in \eqref{eq:ocena_det_det}
can be arbitrarily close to $0$.
However, as explained by Lotz and Noferini \cite{LotzNoferini}, such behavior occurs very rarely
in practice. If we exclude a set of small measure $\delta>0$ from all possible
perturbations such that $\|E\|_*\le 1$, then there exists a constant $C\ge 1$, which
only depends on $m$, $n$, $d$, and $r$, such that
\[
|\lambda_0(\eps)-\lambda_0|\le C\,\delta^{-1} \,\kappa(\lambda_0)\,\eps +{\cal O}(\eps^2),
\]
and $\kappa(\lambda_0)$ is the quantity that describes how ill conditioned an eigenvalue is in this \emph{weak} context.
This is the reason why we will later use an approximation of
\eqref{def:kp} as a criterion whether or not a computed eigenvalue is a finite eigenvalue of
a singular matrix polynomial $P$.

The following definition of a linearization of matrix polynomials can be found, e.g., in \cite{GohLR82}, while the concept of strong linearizations has been
introduced in \cite{GohKL88}.

\begin{definition} A matrix pencil $L(\lambda)=\lambda X + Y$ with
$X,Y\in\CC^{nd\times nd}$ is a \emph{linearization} of an $n\times n$ matrix
polynomial $P(\lambda)$ of degree $d$ if there exist two unimodular
$nd\times nd$ matrix polynomials $E(\lambda)$ and $F(\lambda)$ such that
\[
E(\lambda)\,L(\lambda)\,F(\lambda)=\left[\begin{matrix} P(\lambda) & 0\cr
0 & I_{(d-1)n}\end{matrix}\right].
\]
A linearization $L(\lambda)$ is called a \emph{strong linearization} if
the reversal polynomial
${\rm rev}\,L(\lambda)$ is also a linearization of the
reversal polynomial ${\rm rev}\,P(\lambda)$.
\end{definition}

It follows immediately from the definition that a matrix polynomial and
its linearization have the same Smith form, so the finite eigenvalues and
the corresponding structural indices
of the matrix polynomial are exactly the finite eigenvalues and
corresponding structural indices of its linearization.
If the linearization is strong, then this also holds for the eigenvalue
$\infty$.

We say that a set $\mathcal A\subseteq\mathbb C^n$ is \emph{algebraic} if
it is the set of common zeros of finitely many polynomials. A set
$\Omega\subseteq\mathbb C^n$ is called \emph{generic} if its complement is
contained in an algebraic set that is not the full space $\mathbb C^n$.
\ch{The notion of a \emph{generic} set $\Omega\subseteq\mathbb C^{n\times k}$
 now follows from interpreting $\mathbb C^{n\times k}$ as $\mathbb C^{nk}$.}
We use
the expression that a property holds \emph{generically with respect to the entries
of $U\in\mathbb C^{n\times k}$} if there exists a generic set
$\Omega\subseteq\mathbb C^{n\times k}$ such that the property holds for all
$U\in\Omega$.

\section{Main results}\label{sec:main}
The main results are generalizations of similar results from \cite{HMP_SingGep} and
\cite{HMP_SingGep2}. The corresponding proofs of the results in the pencil case, however,
rely on the existence of the Kronecker canonical form for matrix pencils. Since an
analogous canonical form under strict equivalence is not available for matrix
polynomials, we need different arguments to extend the results to the polynomial case.

\begin{lemma}\label{lem:Lo} Let $P(\lambda)$ be an $n\times n$ singular matrix polynomial of
degree $d$ and normal rank $n-k$, and
let $S_R(\lambda)=\left[x_1(\lambda)\ \dots\ x_{k}(\lambda)\right]$
and $S_L(\lambda)=\left[y_1(\lambda)\ \dots\ y_{k}(\lambda)\right]$
be $n\times k$ matrix polynomials, such that the columns form minimal bases
of, respectively, the right and left nullspaces of
$P(\lambda)$
with left minimal indices $\eta_1, \dotsc,\eta_k$ and right minimal indices $\eps_1, \dotsc,\eps_k$.
Furthermore, let
$U,V\in\mathbb C^{n\times k}$ have full column rank, $N:=\eta_1+\cdots+\eta_k$,
$M:=\eps_1+\cdots+\eps_k$, and let
$\gamma_1, \dotsc,\gamma_{dk}\in\mathbb C$
be given values that are distinct from the eigenvalues of $P(\lambda)$.
Then, generically with respect to the entries of $U$ and $V^*$, the following statements hold:
\begin{enumerate}
\item[\rm 1)] The polynomial $\det\big(V^*S_R(\lambda)\big)$ has exactly $M$ simple roots
$\alpha_1, \dotsc,\alpha_M$ that are different from the eigenvalues of $P(\lambda)$
and different from $\gamma_1, \dotsc,\gamma_{dk}$. For each
$\alpha_i$ there exists a nonzero vector $z_i$ such that $P(\alpha_i)\,z_i=0$ and $V^*z_i=0$.
\item[\rm 2)] The polynomial $\det\big(U^*S_L(\lambda)\big)$ has exactly $N$ simple roots
$\beta_1, \dotsc,\beta_N$ that are different from the eigenvalues of $P(\lambda)$, from
the values $\gamma_1, \dotsc, \gamma_{dk}$, and also from $\alpha_1, \dotsc, \alpha_M$.
For each $\beta_i$ there exists a nonzero vector $w_i$ such that $w_i^*P(\beta_i)=0$ and $w_i^*U=0$.
\item[\rm 3)] For any $dk$ nonzero vectors $t_1, \dotsc,t_{dk}\in\mathbb C^k$ there exist nonzero vectors $s_1, \dotsc,s_{dk}\in\mathbb C^n$ with $P(\gamma_i)\,s_i=0$ and $t_i=V^*s_i$ for $i=1, \dotsc, dk$.
\end{enumerate}
\end{lemma}

\begin{proof}
1) The $k\times k$ matrix $G(\alpha):=V^*S_R(\alpha)$ has elements $g_{ij}(\alpha)=v_i^*x_j(\alpha)$ that are polynomials in $\alpha$ which, generically with respect to the entries of $v_i^*$, have degrees $\eps_j$ for $i,j=1, \dotsc,k$.
It follows that $\det G(\alpha)$ is a polynomial in $\alpha$ which generically with
respect to the entries of $V^*$ is of degree $M$ (cf. the main theorem in
\cite{Forney_minimal_basis}). Thus $\det G(\alpha)=0$ has $M$ roots
$\alpha_1, \dotsc,\alpha_M$ (counted with multiplicities).

For each root $\alpha_i$ there exists a nonzero vector $s_i\in\CC^k$ such
that $G(\alpha_i)\,s_i=0$. We know (see, e.g., \cite[Thm.~2.2]{VanDoorenDopico_Robustness}) that $S_R(\alpha)$ is of full column rank for all $\alpha$.
So, if we take $z_i=S_R(\alpha_i)\,s_i$, then $z_i$ is nonzero and $V^*z_i=V^*S_R(\alpha_i)\,s_i=G(\alpha_i)\,s_i=0$. On the other hand, since
$z_i\in\ker_{\alpha_i}\big(P(\lambda)\big)$, it also holds that $P(\alpha_i)\,z_i=0$.

For a fixed $\mu\in\mathbb C$ we can consider
$\det G(\mu)$ as a polynomial in the entries of $V^*=[v_1 \ \, \dots \ \, v_k]^*$.
For the particular choice $V=S_R(\mu)$
we obtain that $G(\mu)$ is nonsingular, which shows that $\det G(\mu)$ is a nonzero
polynomial in the entries of $V^*$. It follows that $\det G(\mu)\neq 0$
generically with respect to the entries of $V^*$, and consequently the fixed value $\mu$ will generically not be among
the roots of $\det G(\alpha)$ as a polynomial in $\alpha$. Since the intersection of finitely many generic
sets is still generic, it follows that we can exclude finitely many values from the
roots $\alpha_1, \dotsc,\alpha_M$ of $G(\alpha)$. This shows that generically with respect to the entries of $V^*$,
the values $\alpha_1, \dotsc,\alpha_M$ are different from the eigenvalues of $P(\lambda)$
and also from the given values $\gamma_1, \dotsc,\gamma_{dk}$.

It remains to show that the roots $\alpha_1, \dotsc,\alpha_M$ of $\det G(\alpha)$
are generically simple. This is the case if the discriminant of $\det G(\alpha)$,
which is a polynomial in the entries of $V^*$, is nonzero. Thus, to show
that the roots of $\det G(\alpha)$ are generically simple, it remains to show that
this discriminant is not the zero polynomial and to this end it is sufficient to find
one particular example for $V$. The existence of such a $V$ for which the roots of
$\det G(\alpha)$ are simple is guaranteed by Theorem \ref{cnj:need_to_show} in the appendix.

2) Similarly as in 1) we now consider the left null space of $P(\lambda)$ and show
the existence of $\beta_1, \dotsc,\beta_N$ and the corresponding nonzero vectors $w_1, \dotsc, w_N$, where
now the statements are generic with respect to the entries of $U$. In particular, by interpreting $V$
as already fixed, this shows that generically with respect to the entries of $U$, the values
$\beta_1, \dotsc,\beta_N$ are not only different from the eigenvalues of $P(\lambda)$ and
$\gamma_1, \dotsc,\gamma_{dk}$, but also from the values $\alpha_1, \dotsc,\alpha_M$ constructed in $1)$.

3) With the same notation as in 1) we now aim to solve the equations
$G(\gamma_i)\,c_i=t_i$ for $i=1, \dotsc, dk$. Since $\gamma_i$ is different from the
values $\alpha_1, \dotsc,\alpha_M$, we have $\det G(\gamma_i)\neq 0$ and hence $G(\gamma_i)\,c_i=t_i$
is uniquely solvable for nonzero $c_i$ for $i=1, \dotsc, dk$. The vectors that we are looking for are
$s_i=S_R(\gamma_i)\,c_i$ for $i=1, \dotsc, dk$. \qed
\end{proof}

\begin{theorem}\label{thm:main_perturbation}
Let $P(\lambda)=A_0+\lambda A_1+\cdots +\lambda^d A_d$ be an $n\times n$ singular matrix polynomial of degree $d$ and normal rank $n-k$
with right minimal indices $\eps_1\le \cdots \le \eps_{k}$ and
left minimal indices $\eta_1\le \cdots \le \eta_k$ such that
all its eigenvalues are semisimple, and let $M=\eps_1+\cdots+\eps_k$ and
$N=\eta_1+\cdots+\eta_k$. Furthermore, let $B_0,\dots,B_d\in\mathbb C^{k\times k}$ be
such that the matrix polynomial $Q(\lambda):=B_0+\lambda B_1+\cdots+\lambda^dB_d$
is regular and its eigenvalues are simple and distinct from the eigenvalues of
$P(\lambda)$. Then there exists a generic set
$\Omega\subseteq\mathbb C^{n\times k}\times\mathbb C^{k\times n}$ such that for all
$(U,V^*)\in\Omega$ and all $\tau>0$ the matrix polynomial
\begin{equation}\label{eq:pertP}
\widetilde P(\lambda) = P(\lambda) + \tau \, U\,Q(\lambda) \, V^*,\end{equation}
is regular and its $nd$ eigenvalues are independent of $\tau$ and can be classified
into the following four groups:
\begin{enumerate}
 \item[(a)] (True eigenvalues)
 If $\lambda_0$ is an eigenvalue of $P(\lambda)$, then $\lambda_0$ is also
 an eigenvalue of $\widetilde P(\lambda)$. If $x$ and $y$ are corresponding right and left eigenvectors then $V^*x=0$ and $U^*y=0$. There are $(n-k)d-M-N$ such eigenvalues.
 \item[(b)] (Right random eigenvalues) There are $M$ simple eigenvalues, which are all different from the eigenvalues in (a) and those of $Q(\lambda)$, such that $V^*x = 0$ and $U^*y \ne 0$, where $x$ and $y$ are corresponding right and left eigenvectors.
 \item[(c)] (Left random eigenvalues) There are $N$ simple eigenvalues, which are all different from the eigenvalues in (a) and (b) and those of $Q(\lambda)$, such that $V^*x \ne 0$ and $U^*y = 0$, where $x$ and $y$ are corresponding right and left eigenvectors.
 \item[(d)] (Prescribed eigenvalues) If $\lambda_0$ is an eigenvalue of $Q(\lambda)$, \ch{then $\lambda_0$ is also
 an eigenvalue of $\widetilde P(\lambda)$. If} $x$ and $y$ are corresponding right and left eigenvectors, then $V^*x\ne 0$ and $U^*y \ne 0$. There are $kd$ such eigenvalues.
\end{enumerate}
\end{theorem}
\begin{remark}\rm
We use the attribute ``right'' or ``left'' in (a) and (b) to indicate that the
corresponding random eigenvalues can be related to the right or left minimal indices
of the polynomial $P(\lambda)$, respectively.
\end{remark}

\begin{proof}
First, we note that $U$ and $V$ generically have full column rank, so in the
following we assume that this is the case.
Suppose that $\mu_0$ is neither an eigenvalue of $P(\lambda)$ nor $Q(\lambda)$,
which means that $\rank\big(P(\mu_0)\big)=n-k$. Then we have
$\det\big(\widetilde P(\mu_0)\big)\neq 0$ for all $\tau>0$ if and only if
the range of $UQ(\mu_0)V^*$ has a trivial intersection with the kernel of $P(\mu)$.
By the dimension formula, it is easy to construct matrices $U,V\in\mathbb C^{n\times k}$
such that this is the case. Since $\det\big(\widetilde P(\mu_0)\big)$ is a polynomial in the entries of $U$ and $V^*$, this polynomial is not identical to zero and it follows that
generically $\nrank(\widetilde P)=n$ independent of $\tau$.
In such case $\widetilde P(\lambda)$ has $nd$ eigenvalues (some of them can be infinite if $A_d+\tau \, UB_dV^*$ is singular).

(a) Let $\lambda_0$ be a finite eigenvalue of $P(\lambda)$ of multiplicity $r\ge 1$. Then
$\dim\big(\ker(P(\lambda_0))\big)=k+r$ and there exist $r$ linearly independent vectors
$x_1, \dotsc, x_r\in \ker(P(\lambda_0))$ such that $V^*x_i=0$ for $i=1, \dotsc,r$. It
follows that $\widetilde P(\lambda_0)\,x_i=0$ for $i=1, \dotsc,r$ and hence $\lambda_0$
is an eigenvalue of $\widetilde P(\lambda)$ of multiplicity at least $r$ independent
of $\tau>0$. In a similar way we see that there exist
$r$ linearly independent vectors $y_1, \dotsc, y_r\in\ker(P(\lambda_0)^*)$ such that
$U^*y_i=0$ and $y_i^*P(\lambda_0)=y_i^*\widetilde P(\lambda_0)=0$.

To show that (a) holds for $\lambda_0=\infty$ as well, we apply the already proved part
to the reversal ${\rm rev}\,P(\lambda)$.

(b) In the following let $\gamma_1,\dots,\gamma_{dk}$ denote the eigenvalues of
$Q(\lambda)$. Let $S_R(\lambda)$ be an $n\times k$ matrix polynomial such that its
columns form a minimal basis of the right nullspace ${\cal N}_r$
of $P(\lambda)$. By Lemma~\ref{lem:Lo}, $\det(V^*S_R(\alpha))$ generically is a
polynomial of degree $M$ with simple roots $\alpha_1, \dotsc,\alpha_M$ which are all
different from the eigenvalues of $P(\lambda)$ and from $\gamma_1,\dots,\gamma_{dk}$.
If $\alpha_0$ is such \ch{a root}, then by Lemma~\ref{lem:Lo} there exists a nonzero vector
$z$ such that $P(\alpha_0)\,z=0$ and $V^*z=0$. It follows that
$\widetilde P(\alpha_0)\,z=0$, thus $\alpha_0$ is an eigenvalue of
$\widetilde P$ and $z$ is a right eigenvector, independent of $\tau>0$.

Considering now $\tau$ as a variable, it follows that the pencil
\begin{equation}\label{eq:pengh}
G+\tau H:=P(\alpha_0) +\tau \, U\,Q(\alpha_0)\,V^*
\end{equation}
is singular. Clearly, we have $\rank(G)=n-k$ and $\rank(H)=k$, because
$\alpha_0$ differs from all eigenvalues of $P(\lambda)$ and $Q(\lambda)$ and
because $U$ and $V$ have full column rank.

Suppose that \ch{${\rm nrank}(G-\lambda H)=n-j$} for $j\ge 1$, which means that \eqref{eq:pengh} has $j$ right and $j$ left minimal indices.
We know from $Gz=0$ and $Hz=0$ that at least one
right minimal index is equal to zero. The remaining $j-1$ right minimal indices
are all larger than zero, because otherwise there would exist a nonzero $y\in\ker(G)\cap\ker(H)$ linearly independent of $z$. But then $Gy=0$ implies that
$y$ is in the range of $S_R(\alpha_0)$, i.e., there exists
$\widetilde y\in\mathbb C^k$ such that $y=S_R(\alpha_0)\widetilde y$.
Thus, since $Hy=0$, we have that $V^*S_R(\alpha_0)\widetilde y=0$,
because $U$ has full column rank and $Q(\alpha_0)$ is
nonsingular. But this implies $\alpha_0$ would be
a multiple root of $\det V^*S_R(\alpha)$, which is not possible due to Lemma \ref{lem:Lo}.

Now suppose that \eqref{eq:pengh} has a left minimal index zero.
Then there exists a nonzero vector $w$ such that $w^*G=0$ and $w^*H=0$, which implies $w^*U=0$
because $V$ has full rank and $Q(\alpha_0)$ is nonsingular. But then $\alpha_0$
is equal to one of the values $\beta_1, \dotsc,\beta_N$ from Lemma~\ref{lem:Lo}
which is a contradiction to Lemma~\ref{lem:Lo}. Thus, all left minimal
indices of~\eqref{eq:pengh} are larger than or equal to one.

Now $\rank(H)=k$ implies that $\tau_0=\infty$ is an eigenvalue of the pencil
\eqref{eq:pengh} of geometric multiplicity $n-k-j$ and thus algebraic multiplicity
at least $n-k-j$. Similarly, $\tau_0=0$ is an eigenvalue of algebraic multiplicity
at least $k-j$. Therefore, the sum of the multiplicities of the finite and infinite eigenvalues of \eqref{eq:pengh}
is at least $n-2j$. On the other hand, the normal rank of \eqref{eq:pengh} is $n-j$ and the pencil has $j$ left
minimal indices that are all greater or equal to one. It follows from \eqref{eq:IndexSum} of the index sum theorem that
the only possible option is that $j=1$. Therefore, \eqref{eq:pengh} has one left minimal index one, one right minimal index zero, and $0$ and $\infty$ are the only eigenvalues of \eqref{eq:pengh}.

It follows that there exist linearly independent vectors
$w$ and $z$ such that
\begin{equation}\label{eq:left_eigs_alpha}
w^*G=0,\quad z^*G=w^*H\neq 0,\quad z^*H=0,
\end{equation}
which implies $U^*w\ne 0$. Up to scaling, the left eigenvector $y$ of \eqref{eq:pertP}
associated with $\alpha_0$ then has the form
$y(\tau)=w+\tau z$, is a linear function of $\tau$, and satisfies $U^*y(\tau)\ne 0$.

(c) We show this in a similar way as (b).

(d) Let $\gamma_i$ be an eigenvalue of $Q(\lambda)$ with the right eigenvector $t_i\in\CC^k$.
It follows from Lemma \ref{lem:Lo} that there exists a nonzero vector
$x_i$ such that $P(\gamma_i)\,x_i=0$ and $t_i=V^*x_i$, therefore
$\widetilde P(\gamma_i)\,x_i=0$ and $V^*x_i\ne 0$ for $i=1, \dotsc, dk$
independent of $\tau>0$. In a similar way we can find a left eigenvector $y_i$ such
that $y_i^*P(\gamma_i)=0$ and $U^*y_i\ne 0$, again independent of $\tau>0$.
Noting that due to the simplicity of the eigenvalues of $\gamma_i$ left and right
eigenvectors are unique up to multiplication with nonzero scalars concludes the
proof of (d).

\ch{It follows from Theorem~\ref{thm:indexsum} (index sum theorem) that
the sum of the algebraic multiplicities of the eigenvalues of $P(\lambda)$ is $(n-k)d-M-N$.}
By counting the number of eigenvalues of $\widetilde P(\lambda)$, we can see from
the \ch{index sum theorem} that there are no other eigenvalues
than those from (a), (b), (c), and (d)\ch{. In addition,} the algebraic
multiplicity of each of the eigenvalues in (a) are equal for the \ch{polynomials}
$P(\lambda)$ and $\widetilde P(\lambda)$. \qed
\end{proof}

We highlight that the right and left random eigenvalues from Theorem \ref{thm:main_perturbation} depend only on the choices of $U$ and $V$, respectively.
In particular, they do not depend on the choices of the matrix polynomial $Q(\lambda)$
and of $\tau>0$.

\begin{remark}\rm
The corresponding result in the pencil case \cite{HMP_SingGep} has no restrictions on
the multiplicity of the eigenvalues of the singular pencil, because the corresponding
proof is based on the Kronecker canonical form of matrix pencils which allows the
separation of the regular part from the singular part of the matrix pencil. For matrix polynomials there unfortunately is no straightforward extension, and the alternative proof strategy in Theorem~\ref{thm:main_perturbation}
needs the additional assumption that the eigenvalues of the singular pencil are
semisimple. However, numerical experiments suggest that the result also holds for singular
matrix polynomials that have eigenvalues that are not semisimple.
\end{remark}

Let us consider a given singular matrix polynomial $P(\lambda)$ and
matrices $U$, $V$, $B_0, \dotsc,B_k$ as in Theorem~\ref{thm:main_perturbation}. If
$U_\perp$ and $V_\perp$ are matrices whose columns form bases for the orthogonal
complement of the ranges of $U$ and $V$, respectively, then we can build
a strictly equivalent matrix polynomial
\begin{equation}\label{coord}
\widehat P(\lambda):= [U \ \, U_{\perp}]^*\,P(\lambda)\,[V \ \, V_{\perp}]
=
\mtxa{cc}{U^*P(\lambda)\,V & U^*P(\lambda)\,V_{\perp} \\[1mm]
U_{\perp}^*P(\lambda)\,V & U_{\perp}^*P(\lambda)\,V_{\perp}}.
\end{equation}
\ch{If we apply the same equivalence transformation to the perturbed matrix polynomial \eqref{eq:pertP},
we see that $[U \ \, U_{\perp}]^*\,\widetilde P(\lambda)\,[V \ \, V_{\perp}]$
only differs from $\widehat P(\lambda)$
by an extra term $\tau \, Q(\lambda)$ in the $(1,1)$-block.}
The following generalization of
\cite[Prop.~4.1]{HMP_SingGep2} shows that under mild additional assumptions
it is possible to extract the true eigenvalues of the
matrix polynomial $P(\lambda)$ from
the projected matrix polynomial $U_{\perp}^*P(\lambda)\,V_{\perp}$
of size $(n-k)\times (n-k)$.

\begin{proposition}\label{prop_proj}
Let $P(\lambda)$ be a complex $n\times n$ matrix polynomial
of degree $d$ and normal rank $n-k$
such that all its eigenvalues are semisimple. Furthermore, let $U$, $V$, $B_0, \dotsc,B_k$ satisfy the hypotheses of Theorem~\ref{thm:main_perturbation}. If the
$(n-k)\times (n-k)$ matrix polynomial
$P_{22}(\lambda):=U_{\perp}^*P(\lambda)\,V_{\perp}$ from \eqref{coord} is regular, then
its eigenvalues are precisely:
\begin{enumerate}
\item[a)] the random eigenvalues of \eqref{eq:pertP} with the same $U$ and $V$;
\item[b)] the true eigenvalues of $P(\lambda)$.
\end{enumerate}
\end{proposition}
\begin{proof}
We know that the true and random eigenvalues of
$\widetilde P(\lambda)$ are independent of $\tau$ and $B_0, \dotsc,B_k$.
Let $\lambda_0$ be such an eigenvalue. Then it follows from
Theorem~\ref{thm:main_perturbation} that a right eigenvector has the form $V_{\perp} s$
for a nonzero $s\in \CC^{n-k}$ or a left eigenvector is of the form
$U_{\perp} t$ for a nonzero $t\in \CC^{n-k}$ (both statements are true for a true
eigenvalue and exactly one of the statements is true for a random eigenvalue).
It follows that $P_{22}(\lambda_0)\,s=0$ or
$t^*P_{22}(\lambda_0)=0$. Consequently,
$\lambda_0$ is an eigenvalue of $P_{22}(\lambda)$.
Since \eqref{eq:pertP} altogether has $d(n-k)$ eigenvalues from the groups (a), (b), and
(c) from Theorem~\ref{thm:main_perturbation}, it follows by
a simple counting argument that these are all eigenvalues of $P_{22}(\lambda)$.\qed
\end{proof}

\begin{remark}\rm
If we use the slightly weaker concept of genericity with respect to the real and
imaginary parts of the entries of $U$ and $V$ in Proposition~\ref{prop_proj}, then
the regularity of the pencil $U_{\perp}^*P(\lambda)\,V_{\perp}$ is generically
guaranteed by the results from \cite{HMP_SingGep2}: since the normal rank of
$P(\lambda)$ is $n-k$, it follows that the $(n-k)\times (n-k)$
matrix polynomial $\widetilde U^*P(\lambda)\widetilde V$ is regular, generically with
respect to the real and imaginary parts of the entries of $n\times (n-k)$ matrices
$\widetilde U$ and $\widetilde V$. Then using \cite[Prop.~9.3]{HMP_SingGep2} it follows
that $U_{\perp}^*P(\lambda)\,V_{\perp}$ is regular also generically with respect to the
real and imaginary parts of the entries of $U,V\in\CC^{n\times k}$.
\end{remark}

Another option that can be used to compute the true eigenvalues, is to augment a
singular matrix polynomial in such way that it becomes regular while the true eigenvalues remain unchanged. A generalization of the augmented method from \cite{HMP_SingGep2} is to use the $(n+k) \times (n+k)$
augmented matrix polynomial
\begin{equation} \label{augm1x}
P_a(\lambda) :=
\left[ \begin{array}{cc} P(\lambda) & U Q_1(\lambda) \\ Q_2(\lambda)V^* & 0 \end{array} \right],
\end{equation}
where $Q_i(\lambda) = B_0^{(i)}+\lambda B_1^{(i)} +\cdots +\lambda^d B_d^{(i)}$ for $i=1,2$ is a regular $k\times k$ matrix polynomial of degree $d$, and $U,V\in\CC^{n\times k}$. The following result is a generalization of
\cite[Prop. 5.1]{HMP_SingGep2}.

\begin{proposition}\label{prop_augm}
Let $P(\lambda)$ be a complex $n\times n$ matrix polynomial of normal
rank $n-k$ such that all its eigenvalues are semisimple. Assume that the regular $k\times k$ matrix polynomials $Q_1(\lambda)$ and $Q_2(\lambda)$ of degree $d$ are chosen in such a way that their $2dk$ eigenvalues are pairwise distinct.
Furthermore, let $U,V\in \CC^{n\times k}$ have orthonormal columns such that the
augmented matrix polynomial \eqref{augm1x} is regular and $(U,V^*)\in\Omega$ holds,
where $\Omega$ is the generic set from Theorem~\ref{thm:main_perturbation}. Then the matrix polynomial \eqref{augm1x} has the following eigenvalues:
\begin{enumerate}
\item[a)] $2dk$ prescribed eigenvalues, which are precisely the
eigenvalues of $Q_1(\lambda)$ and $Q_2(\lambda)$;
\item[b)] the random eigenvalues of \eqref{eq:pertP} with the same $U$ and $V$;
\item[c)] the true eigenvalues of $P(\lambda)$.
\end{enumerate}
\end{proposition}
\begin{proof}
For part a), if $\mu$ is an eigenvalue of $Q_1(\lambda)$ with
an eigenvector $w \in \CC^k$, then
\[
P_a(\mu)\left[ \begin{array}{cc} 0 \\ w \end{array} \right]=0
\]
and $\mu$ is an eigenvalue of \eqref{augm1x}. In a similar way,
if $\mu$ is an eigenvalue of $Q_2(\lambda)$ with
a left eigenvector $z \in \CC^k$, then
\[
[\, 0 \ \ z^*\,] \ P_a(\mu)=0
\]
and $\mu$ is an eigenvalue of \eqref{augm1x}. As the eigenvalues of $Q_1(\lambda)$ and
$Q_2(\lambda)$ are pairwise distinct, this gives $2dk$ eigenvalues of
$P_a(\lambda)$.

For b) and c) we consider the perturbed pencil \eqref{eq:pertP} with the same $U$ and $V$ and an arbitrary $Q(\lambda)$ satisfying the hypothesis of Theorem~\ref{thm:main_perturbation} exploiting that the
random eigenvalues of \eqref{eq:pertP} are independent of $\tau$ and $Q(\lambda)$.
Let $U_\perp$ and $V_\perp$ be matrices whose columns form bases for the orthogonal
complement of the ranges of $U$ and $V$, respectively.
If $\mu$ is a random eigenvalue of \eqref{eq:pertP} then it follows from Theorem \ref{thm:main_perturbation} that either
the corresponding right eigenvector has the form $V_{\perp} s$ for
a nonzero $s\in \CC^{n-k}$ or the corresponding left eigenvector
has the form $U_{\perp} t$ for a nonzero $t\in \CC^{n-k}$. Then either
\[
P_a(\mu)\left[ \begin{array}{cc} V_{\perp} s \\ 0 \end{array} \right]=0\quad{\rm or}\quad
[\, t^*U_{\perp}^* \ \ 0 \,] \ P_a(\mu)=0
\]
and $\mu$ is thus an eigenvalue of $P_a(\lambda)$.

For c), similarly as in b), we know from Theorem \ref{thm:main_perturbation} that $\mu$ is a true eigenvalue of
$P(\lambda)$ when $\mu$ is an eigenvalue
of \eqref{eq:pertP} with
the right and left eigenvector
of the form $V_{\perp} s$ and $U_{\perp} t$,
where $s,t\in \CC^{n-k}$.
It follows that
\[
P_a(\mu)\left[ \begin{array}{cc} V_{\perp} s \\ 0 \end{array} \right]=0\quad{\rm and}\quad
 [\, t^*U_{\perp}^* \ \ 0 \,] \ P_a(\mu)=0,
\]
therefore $\mu$ is an eigenvalue of
$P_a(\lambda)$.

As all eigenvalues in c) are semisimple and all eigenvalues in a) and b) are simple, it follows by a counting argument that these are all the eigenvalues of $P_a(\lambda)$.\qed
\end{proof}

From the above proof it is easy to see how to extract the
true eigenvalues of $P(\lambda)$ from the eigenvalues
of $P_a(\lambda)$.
If $(\theta, x, y)$ is an eigentriplet of
$P_a(\lambda)$, where
\[
x = \left[ \begin{array}{cc} x_1 \\ x_2 \end{array} \right],\quad
y = \left[ \begin{array}{cc} y_1 \\ y_2 \end{array} \right] \]
are in block form in accordance with \eqref{augm1x}, then
$\theta$ is a true eigenvalue if and only if
$x_2=y_2=0$. Here we assume that the matrix polynomials $Q_1(\lambda)$ and
$Q_2(\lambda)$ are chosen in such way that all prescribed eigenvalues differ from the true eigenvalues of $P(\lambda)$.

\ch{Compared to the other two approaches, the augmented approach leads to a regular PEP with the largest matrices, but has the advantage that it is more suitable for problems with sparse matrices, which has recently been exploited for singular matrix pencils
in \cite{Meerbergen_Wang_Arxiv}.}

\section{Three methods}\label{sec:method} In this section we
present three numerical methods for singular PEPs that are generalizations of methods
for generalized eigenvalue problems from \cite{HMP_SingGep} and \cite{HMP_SingGep2}.
All methods return the regular eigenvalues of a singular PEP together with the reciprocals of the condition numbers.
Based on the reciprocals of the condition numbers and the gaps between the eigenvalues we can then classify the
regular eigenvalues into finite and infinite ones. This is done in Algorithm~4 that is based on the heuristic approach
proposed in \cite[Alg.~3]{HMP_SingGep2}, for more details see \cite[Sec.~6]{HMP_SingGep2}.

The first method is based on Theorem~\ref{thm:main_perturbation} and is
a generalization of \cite[Alg.~1]{HMP_SingGep}. In this method we perturb a
singular PEP into a regular PEP of the same size. In all algorithms, $\eps$ denotes the machine precision.

\noindent\vrule height 0pt depth 0.5pt width \textwidth \\
{\bf Algorithm~1: Eigenvalues of a singular PEP by a rank-completing perturbation}. \\[-3mm]
\vrule height 0pt depth 0.3pt width \textwidth \\
{\bf Input:} $A_0, \dotsc, A_d\in \CC^{n\times n}$, $k = n-\text{nrank}(P)$, perturbation constant
$\tau$ (default $10^{-2}$),
threshold $\delta$ (default $\eps^{1/2}$).\\
{\bf Output:} Eigenvalues of $P(\lambda)=A_0 +\cdots +\lambda^d A_d$ with reciprocals of the condition numbers.\\
\begin{tabular}{ll}
{\footnotesize 1:} & Select random $n\times k$ matrices $U$ and $V$ with
orthonormal columns.\\
{\footnotesize 2:} & Select random $k\times k$ matrices $B_0, \dotsc,B_d$.\\
{\footnotesize 3:} & Compute the eigenvalues $\lambda_i$, $i=1, \dotsc, dn$, and normalized right and left\\
& eigenvectors $x_i$ and $y_i$ of
$\widetilde P(\lambda)=P(\lambda) + \tau \, U(B_0+\lambda B_1 +\cdots +\lambda^d B_d)\,V^*$.\\
{\footnotesize 4:} & Compute $\alpha_i=\|V^*x_i\|$, $\beta_i=\|U^*y_i\|$,
$i=1, \dotsc, dn$. \\
{\footnotesize 5:} & Compute $\gamma_i=|y_i^*P'(\lambda_i)\,x_i| \cdot (1+|\lambda_i|^2+\cdots+|\lambda_i|^{2d})^{-1/2}$ for $i=1, \dotsc, dn$.\\
{\footnotesize 6:} & Return $\lambda_i$ and $\gamma_i$ for those
$i=1, \dotsc, dn$ such that $\max(\alpha_i,\beta_i) < \delta $.
\end{tabular} \\
\vrule height 0pt depth 0.5pt width \textwidth
\medskip

The second method, which is based on Proposition~\ref{prop_proj}, is a generalization of \cite[Alg.~1]{HMP_SingGep2}. In this method we project an $n\times n$ singular PEP
of normal rank $n-k$ to an $(n-k)\times(n-k)$ regular PEP.

\noindent\vrule height 0pt depth 0.5pt width \textwidth \\
{\bf Algorithm~2: Eigenvalues of a singular PEP by projection}. \\[-3mm]
\vrule height 0pt depth 0.3pt width \textwidth \\
{\bf Input:} $A_0, \dotsc,A_d\in \CC^{n\times n}$, $k = n-\text{nrank}(P)$,
threshold $\delta$ (default $\eps^{1/2}$).\\
{\bf Output:} Eigenvalues of $P(\lambda)=A_0 +\cdots +\lambda^d A_d$ with reciprocals of the condition numbers. \\
\begin{tabular}{ll}
{\footnotesize 1:} & Select random unitary $n\times n$ matrices $\widehat W$ and $\widehat Z$, and partition them
as\\
{} & $\widehat W = [W \ \, W_\perp]$ and $\widehat Z=[Z \ \, Z_\perp]$,
where $W$ and $Z$ have $n-k$ columns.\\
{\footnotesize 2:} & Compute the eigenvalues $\lambda_i$, $i=1, \dotsc, d(n-k)$, and normalized right and left\\
& eigenvectors $x_i$ and $y_i$ of
$\widetilde P(\lambda)=W^*A_0Z+\lambda \,W^*A_1Z +\cdots +\lambda^d \, W^*A_dZ$.\\
{\footnotesize 3:} & Compute $\alpha_i=\|W_\perp^*P(\lambda_i) Zx_i\|$, $\beta_i=\|y_i^*W^*P(\lambda_i)\,Z_\perp\|$,
$i=1, \dotsc, d(n-k)$. \\
{\footnotesize 4:} & Compute $\gamma_i=|y_i^*W^*P'(\lambda_i)\,Zx_i|\,(1+|\lambda_i|^2+\cdots+|\lambda_i|^{2d})^{-1/2}$ for $i=1, \dotsc, d(n-k)$.\\
{\footnotesize 5:} & Return $\lambda_i$ and $\gamma_i$ for those
$i=1, \dotsc, d(n-k)$ such that\\
& $\max(\alpha_i,\beta_i) < \delta \, (\,\|A_0\|+|\lambda_i|\,\|A_1\|+\cdots+|\lambda_i|^d\,\|A_d\|\,)$.
\end{tabular} \\
\vrule height 0pt depth 0.5pt width \textwidth
\medskip

The third and final approach, based on Proposition~\ref{prop_augm}, is an extension of \cite[Alg.~2]{HMP_SingGep2}. Here we augment an $n\times n$ singular PEP
of normal rank $n-k$ to an $(n+k)\times(n+k)$ regular PEP.

\noindent\vrule height 0pt depth 0.5pt width \textwidth \\
{\bf Algorithm~3: Eigenvalues of a singular PEP by augmentation}. \\[-3mm]
\vrule height 0pt depth 0.3pt width \textwidth \\
{\bf Input:} $A_0, \dotsc,A_d\in \CC^{n\times n}$, $k = n-\ch{\text{nrank}(P)}$,
threshold $\delta$ (default $\eps^{1/2}$).\\
{\bf Output:} Eigenvalues of $P(\lambda)=A_0 +\cdots +\lambda^d A_d$ with reciprocals of the condition numbers.\\
\begin{tabular}{ll}
{\footnotesize 1:} & Select random $n\times k$ matrices $U$ and $V$ with orthonormal columns.\\
{\footnotesize 2:} & Select random $k\times k$ matrices $B_0^{(i)}, \dotsc,B_d^{(i)}$ for $i=1,2$. \\
{\footnotesize 3:} & Compute the eigenvalues $\lambda_i$, $i=1,\dots,d(n+k)$, and normalized right and left \\
& eigenvectors $\smtxa{c}{x_{i1} \\ x_{i2}}$ and $\smtxa{c}{y_{i1} \\ y_{i2}}$ of the augmented matrix polynomial \eqref{augm1x}.
\\
{\footnotesize 4:} & Compute $\alpha_i=\|x_{i2}\|$, \ $\beta_i=\|y_{i2}\|$,
\ $i=1,\dots,d(n+k)$. \\
{\footnotesize 5:} & Compute $\gamma_i=|y_{i1}^*P'(\lambda_i)\,x_{i1}| \cdot (1+|\lambda_i|^2+\cdots+|\lambda_i|^{2d})^{-1/2}$, \ $i=1,\ldots, d(n+k)$.\\
{\footnotesize 6:} & Return $\lambda_i$ and $\gamma_i$ for those
$i=1,\ldots,d(n+k)$ such that $\max(\alpha_i,\beta_i)<\delta$.
\end{tabular} \\
\vrule height 0pt depth 0.5pt width \textwidth
\medskip

To classify a computed regular eigenvalue $\lambda_i$ as finite or infinite,
we use the computed reciprocal of the condition number $\gamma_i$ combined with
the relative gap of an eigenvalue, which is defined as
\[
{\rm gap}_i=\min_{j\ne i}\frac{|\lambda_j-\lambda_i|}{(1+|\lambda_i|^2)^{1/2}}.
\]
The criterion is based on the fact that $y_i^*P'(\lambda_i)\,x_i\ne 0$ if
$\lambda_i$ is a simple finite eigenvalue. A singular PEP can also have multiple
eigenvalues for which $\gamma_i$ might be $0$, yet in practice,
as explained in more details in \ch{\cite[Section 6]{HMP_SingGep2}}, an eigenvalue of multiplicity $m$
is numerically evaluated as $m$ simple eigenvalues with nonzero values $\gamma_i$ and this
way multiple finite eigenvalues are properly identified in most cases.
We use the relative gap to further improve the detection when
$\gamma_i$ is small. We expect that a computed representative $\lambda_i$ of a multiple finite
eigenvalue will have
a small $\gamma_i$ but also a small
${\rm gap}_i$. On the other hand, a (multiple) infinite eigenvalue will usually
appear as a finite eigenvalue with
$\gamma_i$ very close to zero and a large ${\rm gap}_i$ (close to $1$).

\ch{The computed values $\gamma_i$ in Algorithms~1--3 are approximations
for the values $\gamma(\lambda_i)$ from Definition \ref{df:kappa}, i.e., the condition numbers of the computed eigenvalues as eigenvalues of the transformed (perturbed, projected or augmented) PEP are approximations for the weak condition numbers of
the initial singular PEP.

Following the same arguments as in \cite{DanielBor_BIT}, where this analysis is done for singular pencils, one can show that
$\gamma_i=|\alpha_i|\,|\beta_i|\gamma(\lambda_i)$, where
$0\le |\alpha_i|,|\beta_i|\le 1$ and $\alpha_i,\beta_i$ depend on the choice of random $n\times k$ matrices $U$ and $V$.
If $U$ and $V$ are obtained from QR decompositions of independent complex Gaussian random matrices, then \cite[Prop.~3]{DanielBor_BIT} shows that
$|\alpha_i|$ and $|\beta_i|$ are independent random variables
and $|\alpha|^2,|\beta|^2\sim \mathrm{Beta}\left(1,k \right)$,
where $\mathrm{Beta}$ denotes the beta distribution.
It follows that with high
probability the sensitivity does not increase very much, i.e., the product $|\alpha_i|\cdot|\beta_i|$ is not very small.
\cm{It is thus unlikely that the eigenvalue condition numbers of the transformed PEP are much larger than
the weak eigenvalue condition numbers of the original singular PEP. Hence, it is safe to use
the computed eigenvalue conditions numbers to detect simple well-conditioned finite eigenvalues.}}

Algorithm~4 is based on \cite[Alg.~3]{HMP_SingGep2}, with slightly different default values.
Its input are regular eigenvalues and reciprocals of condition numbers that we get from any of
Algorithms~1--3.

\noindent\vrule height 0pt depth 0.5pt width \textwidth \\
{\bf Algorithm~4: Extraction of finite eigenvalues from the regular eigenvalues
of a singular PEP}. \\[-3mm]
\vrule height 0pt depth 0.3pt width \textwidth \\
{\bf Input:} Eigenvalues $\lambda_i$
 and reciprocals $\gamma_i$ of condition numbers, $i=1,\ldots,r$, for regular eigenvalues of PEP,
thresholds $\delta_1$ (default $\varepsilon$), $\delta_2$ (default $10^4\,\varepsilon$),
and $\xi_2$ (default 0.01).\\
{\bf Output:} Finite eigenvalues. \\
\begin{tabular}{ll}
{\footnotesize 1:} & Compute ${\rm gap}_{i}=\min_{j\ne i} |\lambda_j-\lambda_i| \cdot (1+|\lambda_i|^2)^{-1/2}$, \ $i=1,\ldots,r$.\\
{\footnotesize 2:} & If $\gamma_i<\delta_1$, \ $i=1,\ldots,r$, flag $\lambda_i$ as an infinite eigenvalue.\\
{\footnotesize 3:} & If $\gamma_i<\delta_2$ and ${\rm gap}_i>\xi_2$, \ $i=1,\ldots,r$, flag $\lambda_i$ as an infinite eigenvalue.\\
{\footnotesize 4:} & Return $\lambda_i$ for those $i=1,\ldots,r$ such that $\lambda_i$ is not flagged as an infinite eigenvalue.\\
\end{tabular} \\
\vrule height 0pt depth 0.5pt width \textwidth
\medskip

\section{Comparison to linearization}\label{sec:compare_lin}
A common way to solve a PEP is to apply one of the numerous possible linearizations. If a linearization is strong, then from
a singular matrix polynomial $P(\lambda)$ we get a singular matrix pencil
$L(\lambda):=E+\lambda F$ such that the finite and infinite eigenvalues of $L$ and $P$ coincide. Suppose that
$P$ is an $n\times n$ matrix polynomial of degree $d$ and normal rank $n-k$.
We can apply the following two procedures to compute the eigenvalues of $P$.
\begin{enumerate}
 \item[1)] First, apply Algorithm~2 and project $P(\lambda)$ into
 a regular $(n-k)\times (n-k)$ matrix polynomial $\widetilde P(\lambda)$
 of degree $d$. Then, to compute the eigenvalues and eigenvectors of $\widetilde P$, linearize it
 into a regular matrix pencil $L_1(\lambda)$ of size $d(n-k)\times d(n-k)$.
 \item[2)] First, linearize $P(\lambda)$ into a singular
 matrix pencil $L(\lambda)$ of size $dn\times dn$. Then, apply Algorithm~2 to $L$ and project it into a regular matrix pencil
 $L_2(\lambda)$.
\end{enumerate}

Suppose for example that in 2) we use a strong linearization from the space
$\mathbb L_1(P)$ introduced in \cite{MacMMM06}.
If $\eps_1\le \cdots \le\eps_k$ are the right minimal indices of
$P(\lambda)$, then by \cite[Thm.~5.10]{DeTeranDopicoMackey09}
the right minimal indices of $L(\lambda)$ are
$\eps_1+d-1\le\cdots\le \eps_k+d-1$, while the left minimal
indices of $P(\lambda)$ and $L(\lambda)$ are equal. It then follows from
the index sum theorem that the normal rank of
$L(\lambda)$ is $nd-k$ and thus $L_2(\lambda)$ is of size
$(nd-k)\times (nd-k)$. The results in \cite{DeTeranDopicoMackey09} show
that also types of strong linearizations other than from the space
$\mathbb L_1(P)$ will lead to projected pencils of the same size.

This shows that the two approaches are
not equivalent and if we apply 1) we end up with
a smaller regular pencil. This requires less computational work, in particular when
$d$ is large and the normal rank of $P(\lambda)$ is small. Also, in both
combinations we have to determine the normal rank at a certain point, which we
compute from the
rank of $P(\lambda)$ or $L(\lambda)$ evaluated at random values
$\lambda$. While in 1) we compute it from a rank of an $n\times n$ matrix, in 2) we
have to work with $nd\times nd$ matrices.

Based on the above, it seems to be more efficient to first project the polynomial into a regular one and then linearize it.

\section{Applications and numerical examples}\label{sec:num_results}
We test our numerical methods for singular PEPs on problems coming from several applications.
Two motivating problems are certain bivariate polynomial systems and \emph{ZGV points}, which may not yet be widely known.
We also treat some problems from \cite{DanielIvana} and other papers.
In addition, although there are additional references that mention the need of solving singular PEPs, they usually do not provide concrete examples. We hope that this paper may also encourage new research that will lead to additional applications that involve eigenvalues of singular matrix polynomials.

All numerical experiments have been carried out in MATLAB R2023a using standard double precision. To compute eigenvalues and eigenvectors of
PEPs in Algorithms~1--3 we use {\sf polyeig} in Matlab. We modified the method so that it
returns the left eigenvectors (which are computed in the original {\sf polyeig} for the estimation
of condition numbers but are not returned) as well.
For QEPs we use {\sf quadeig} from \cite{Quadeig}, which is more efficient and accurate than {\sf polyeig} and returns both right and left eigenvectors.

Inspired by \cite{DanielIvana} we performed $N=10000$ runs of Algorithms 1--3, followed by Algorithm~4, for each of the problems and counted the number of successful runs $M$ when the method returned the expected number of finite eigenvalues.
The ratio $p=M/N$ is the empirical probability that the algorithm returns a correct answer. For problems, for which we know the exact solution, we also computed the maximal absolute error of the computed eigenvalues (in all successful attempts) and we report this as the maximal error.
We used the threshold parameter $\delta=10^{-10}$ for Algorithm~1 and $\delta=10^{-12}$ for Algorithms~2 and 3. In Algorithm 4 we used the default values for $\delta_1,\delta_2$, and $\xi_2$, \ch{i.e.,
$\delta_1=10^{-16}$, $\delta_2=10^{-12}$, and $\xi_2=10^{-2}$}.
Instead of fine-tuning the parameters for each \ch{numerical example individually, we empirically selected values that perform well across all examples.}

\subsection{Polynomial equations}
Consider a system of equations involving bivariate polynomials of the form
\[
p(\lambda^2, \mu) = 0, \qquad q(\lambda, \mu) = 0,
\]
where in all monomials of the first polynomial the degree of $\lambda$ is even.
The uniform determinantal representations of \cite{BDDHP_Uniform}
render $A_i$, $B_i$, $C_i$, for $i=1,2$, such that
\begin{align*}
p(\lambda^2, \mu) & = \det(A_1 + \lambda^2 \, B_1 + \mu \, C_1), \\
q(\lambda, \mu) & = \det(A_2 + \lambda \, B_2 + \mu \, C_2),
\end{align*}
and the associated two-parameter eigenvalue problem
\begin{align*}
(A_1 + \lambda^2 \, B_1 + \mu \, C_1)\,x & = 0, \\
(A_2 + \lambda \, B_2 + \mu \, C_2)\,y & = 0.
\end{align*}
If we take the Kronecker product of the first equation and $C_2y$ and
subtract the Kronecker product of $C_1x$ and the second equation,
we obtain the QEP
\begin{equation}\label{eq:michiel_lambda2}
(A_1 \otimes C_2 - C_1 \otimes A_2)\,z - \lambda \, (C_1 \otimes B_2)\,z + \lambda^2 \, (B_1 \otimes C_2)\,z =0,
\end{equation}
where $z=x\otimes y$.
This QEP is typically singular. We give a simple example.

\begin{example}\label{ex:michiel2} \rm
Consider the polynomial system
\begin{align*}
p(\lambda^2,\mu) & = 1+2\lambda^2 + 3\mu + 4\lambda^4 + 5\lambda^2\mu + 6\mu^2, \\
q(\lambda,\mu) & = 6+5\lambda + 4\mu + 3\lambda^2 + 2\lambda\mu + \mu^2,
\end{align*}
which has 8 finite solutions
\begin{align*}
(\lambda_{1,2}, \, \mu_{1,2})&=(-0.658067\pm 0.750641i,\ -0.816143\pm 0.414507i), \\
(\lambda_{3,4}, \, \mu_{3,4})&=(-1.332648\pm 0.355433i,\ -0.287759\pm 1.672047i), \\
(\lambda_{5,6}, \, \mu_{5,6})&=(\phantom-0.475211\pm 1.902116i,\ -0.097155\mp 3.062254i), \\
(\lambda_{7,8}, \, \mu_{7,8})&=(\phantom-2.765503\pm 0.580944i,\ -5.548943\pm 3.891223i).
\end{align*}
Then
\begin{align*}
p(\lambda^2,\mu) &= -\,\left|
\smtxa{rrr}{0 & 2 & 1\\ 3 & 1 & 0 \\ 1 & 0 & 0} + \lambda^2 \, \smtxa{rrr}{0 & 4 & 0 \\ 5 & 0 & -1\\ 0 & 0 & 0} +
\mu \, \smtxa{rrr}{0 & 0 & 0 \\ 6 & 0 & 0\\ 0 & -1 & 0}
\right|\ , \\
q(\lambda,\mu) &= -\,\left|
\smtxa{rrr}{0 & 5 & 1\\ 4 & 6 & 0 \\ 1 & 0 & 0} + \lambda \, \smtxa{rrr}{0 & 3 & 0 \\ 2 & 0 & -1\\ 0 & 0 & 0} +
\mu \, \smtxa{rrr}{0 & 0 & 0 \\ 1 & 0 & 0\\ 0 & -1 & 0}
\right|.
\end{align*}
The corresponding QEP \eqref{eq:michiel_lambda2} is
\begin{equation}\label{eq:ex:michiel2}
P(\lambda)=
\smtxa{ccccccccc}{
0 & 0 & 0 & 0 & 0 & 0 & 0 & 0 & 0 \\
0 & 0 & 0 & 2 + \lambda^2 & 0 & 0 & 1 & 0 & 0 \\
0 & 0 & 0 & 0 & -2-4\lambda^2 & 0 & 0 & -1 & 0 \\
0 & -30-18\lambda & -6 & 0 & 0 & 0 & 0 & 0 & 0 \\
-21-12\lambda+5\lambda^2 & -36 & 6\lambda & 1 & 0 & 0 & -\lambda^2 & 0 & 0 \\
-6 & -3-5\lambda^2 & 0 & 0 & -1 & 0 & 0 & \lambda^2 & 0 \\
0 & 0 & 0 & 0 & 5+3\lambda & 1 & 0 & 0 & 0 \\
1 & 0 & 0 & 4+2\lambda & 6 & -\lambda & 0 & 0 & 0 \\
0 & -1 & 0 & 1 & 0 & 0 & 0 & 0 & 0},
\end{equation}
with normal rank equal to $8$ so that $P(\lambda)$ is a singular QEP.
If we apply Algorithm~2, we obtain the finite eigenvalues in \ch{Table~\ref{tab:michiel2}} which agree with the $\lambda$-components of the solution.

\begin{table}[htb!]
\centering
\caption{Results of Algorithm~2 (projection to normal rank) applied to the singular QEP \eqref{eq:ex:michiel2} with matrices from Example~\ref{ex:michiel2}.\label{tab:michiel2}}
{\footnotesize \begin{tabular}{c|clllll} \hline \rule{0pt}{2.3ex}%
$j$ & $i k_j$ & $\quad \ \gamma_j$ & $\quad \ \alpha_j$ & $\quad \ \beta_j$ & ${\rm gap}_j$ & Type \\[0.5mm]
\hline \rule{0pt}{2.5ex}%
1 & $-0.658067 + 0.750641i$ & $3.5\cdot 10^{-2}$ & $7.5\cdot 10^{-17}$ & $4.2\cdot 10^{-17}$ & $0.24$ & Finite true \\
2 & $-0.658067 - 0.750641i$ & $1.8\cdot 10^{-2}$ & $3.9\cdot 10^{-17}$ & $3.9\cdot 10^{-17}$ & $0.24$ & Finite true \\
3 & $-1.332648 + 0.355434i$ & $9.1\cdot 10^{-3}$ & $6.5\cdot 10^{-17}$ & $3.7\cdot 10^{-17}$ & $0.21$ & Finite true \\
4 & $-1.332648 - 0.355434i$ & $4.6\cdot 10^{-3}$ & $2.7\cdot 10^{-17}$ & $2.5\cdot 10^{-17}$ & $0.21$ & Finite true \\
5 & $\ph-0.475211 + 1.902116i$ & $3.2\cdot 10^{-3}$ & $4.0\cdot 10^{-17}$ & $2.3\cdot 10^{-17}$ & $0.44$ & Finite true \\
6 & $\ph-0.475211 - 1.902116i$ & $2.6\cdot 10^{-3}$ & $5.4\cdot 10^{-17}$ & $2.8\cdot 10^{-17}$ & $0.44$ & Finite true \\
7 & $\ph-2.765503 + 0.580944i$ & $9.0\cdot 10^{-4}$ & $1.5\cdot 10^{-17}$ & $1.6\cdot 10^{-17}$ & $0.28$ & Finite true \\
8 & $\ph-2.765503 - 0.580944i$ & $8.8\cdot 10^{-4}$ & $6.1\cdot 10^{-17}$ & $1.6\cdot 10^{-17}$ & $0.28$ & Finite true \\
9 & $\ph-2.19328\cdot 10^7+2.41952\cdot 10^6i$ & $8.8\cdot 10^{-24}$ & $9.5\cdot 10^{-17}$ & $9.7\cdot 10^{-17}$ & $0.35$ & Infinite true \\
10 & $-2.19328\cdot 10^7-2.41952\cdot 10^6i$ & $4.6\cdot 10^{-24}$ & $1.1\cdot 10^{-16}$ & $1.2\cdot 10^{-16}$ & $0.35$ & Infinite true \\
11 & $\ph-1.91184\cdot 10^7-4.69967\cdot 10^6i $ & $7.3\cdot 10^{-24}$ & $1.7\cdot 10^{-16}$ & $2.6\cdot 10^{-17}$ & $0.39$ & Infinite true \\
12 & $-1.91184\cdot 10^7+4.69967\cdot 10^6i $ & $4.3\cdot 10^{-24}$ & $1.8\cdot 10^{-16}$ & $7.8\cdot 10^{-17}$ & $0.39$ & Infinite true \\
13 & $\infty$ & $0.0$ & $9.6\cdot 10^{-17}$ & $1.3\cdot 10^{-16}$ & $1.00$ & Infinite true \\
14 & $\infty$ & $0.0$ & $1.5\cdot 10^{-16}$ & $5.6\cdot 10^{-17}$ & $1.00$ & Infinite true \\
15 & $\infty$ & $0.0$ & $5.6\cdot 10^{-17}$ & $6.2\cdot 10^{-17}$ & $1.00$ & Infinite true \\
16 & $\infty$ & $0.0$ & $1.4\cdot 10^{-17}$ & $5.9\cdot 10^{-17}$ & $1.00$ & Infinite true \\
\hline
\end{tabular}}
\end{table}

\ch{We remark that the last four eigenvalues in Table~\ref{tab:michiel2} are identified as infinite eigenvalues already during the deflation process in \texttt{quadeig}. More sophisticated deflation methods, such as, e.g., in \texttt{kvadaig} \cite{DrmacSainGlibic2020} for quadratic and \texttt{kvarteig} \cite{DrmacSainGlibic2022} for quartic eigenvalue problems, can
reliably extract even more infinite eigenvalues. While in principle one
can use an arbitrary numerical method for regular PEPs in Algorithms~1--3
as long as the method returns
eigenvalues together with right and left eigenvectors, one could benefit from
using an advanced method tailored for polynomials having particular properties. By using additional results that such a method provides, such as, e.g., the condition numbers, it should be possible to further improve both the separation of true and fake eigenvalues as well as the extraction of finite eigenvalues.
}

We \ch{run} Algorithms~1--3 for 10000 times on \eqref{eq:ex:michiel2} and compute the empirical probability of success $p_i$ and maximal error ${\rm maxerr}_i$ for $i=1,2,3$. The results are $p_1=p_2=p_3=1$,
${\rm maxerr}_1=1.0\cdot 10^{-11}$, ${\rm maxerr}_2=7.6\cdot 10^{-11}$, and
${\rm maxerr}_3=1.3\cdot 10^{-10}$.

In Table~\ref{tab:polyrank} we give the generic normal ranks corresponding to some low values of the total degree $d$ (in $\lambda^2$ and $\mu$ for $p$; in $\lambda$ and $\mu$ for $q$).
\medskip
\begin{table}[htb!]
\centering
\footnotesize
\caption{\ch{Observed normal ranks from singular pencils arising in determinantal representations from generic bivariate polynomials of degree $d$.}\label{tab:polyrank}}
\smallskip
\begin{tabular}{cccc} \hline \rule{0pt}{2.5ex}%
Degree $d$ & $n=(2d-1)^2$ & nrank & Missing \\ \hline \rule{0pt}{2.3ex}%
 2 & \ph{11}9 & \ph18 & \ph11 \\
 3 & \ph{1}25 & 21 & \ph14 \\
 4 & \ph{1}49 & 40 & \ph19 \\
 5 & \ph{1}81 & 65 & 16 \\
 6 & 121 & 96 & 25 \\
\hline
\end{tabular}
\end{table}
\medskip
For $n = (2d-1)^2$, the normal rank seems to be $\text{nrank} = d\,(3d-2)$, with missing rank $(d-1)^2$.
Therefore, this is a problem of dimension $\sim 4d^2$, with normal rank $\sim 3d^2$ and rank-completing dimension $\sim d^2$.
\ch{These normal ranks are observed in experiments.
A detailed study of the structure of the pencils, which might give evidence of this, is outside the scope of this paper.}
\end{example}

\subsection{ZGV points}\label{subs:ZGV}
In the study of anisotropic elastic waveguides (see, e.g., \cite{prada_local_2008}) we
obtain an eigenvalue problem
\begin{equation}\label{eq:ZGV1}
\big((ik)^2 L_2+ ik L_1+L_0+\omega^2 M\big)\,u=0,
\end{equation}
where $L_0$, $L_1$, $L_2$, $M$ are real $n\times n$ matrices obtained by a discretization
of a boundary value problem. The solution
 are dispersion curves $\omega=\omega(k)$, where
the angular frequency $\omega$ of a guided wave is related to the wavenumber $k$.
We are interested in the {\em zero-group-velocity} (ZGV) points on dispersion curves, where
$\omega$ and $k$ are real and
$\omega'(k)=0$. If we assume that $u=u(k)$ and $\omega(k)$ are differentiable, we obtain from
\eqref{eq:ZGV1} (see \cite{ZGV_JASA_23} for details) that at a ZGV point $(\omega,k)$ we have
\begin{equation}\label{eq:ZGV2}
\big((ik)^2\, \widetilde L_2+ ik\, \widetilde L_1+\widetilde L_0+\omega^2\, \widetilde M\big)\,\widetilde u=0,
\end{equation}
where
\[
\widetilde{L}_2=\left[\begin{matrix}L_2 & 0 \\ 0 & L_2\end{matrix}\right],\quad
\widetilde{L}_1=\left[\begin{matrix}L_1 & 0 \\ 2L_2 & L_1\end{matrix}\right],\quad
\widetilde{L}_0=\left[\begin{matrix}L_0 & 0 \\ L_1 & L_0\end{matrix}\right],\quad
\widetilde{M}=\left[\begin{matrix}M & 0 \\ 0 & M\end{matrix}\right],\quad
\widetilde{u}=\left[\begin{matrix}u \\ u'\end{matrix}\right].
\]
Equations \eqref{eq:ZGV1}--\eqref{eq:ZGV2} form
a quadratic two-parameter eigenvalue problem \cite{MP_Q2EP, HMP_LinQ2MEP}.
If we take the Kronecker product of \eqref{eq:ZGV1} and $\widetilde M \widetilde u$, and subtract the Kronecker product of $Mu$ and \eqref{eq:ZGV2}, then we get a $2n^2\times 2n^2$ QEP
\begin{equation}\label{eq:ZGV_QEP}
\big((ik)^2\,\Gamma_2 + ik\, \Gamma_1 +\Gamma_0\big)\,z=0,
\end{equation}
where $z=u\otimes \widetilde u$ and
\[
\Gamma_2=L_2\otimes \widetilde M - \widetilde M\otimes L_2,\quad
\Gamma_1=L_1\otimes \widetilde M - \widetilde M\otimes L_1,\quad
\Gamma_0=L_0\otimes \widetilde M - \widetilde M\otimes L_0 .
\]
With respect to the entries of the matrices $L_0$, $L_1$, $L_2$, $M$, the normal rank
of the quadratic matrix polynomial
\eqref{eq:ZGV_QEP} generically is $2n^2-n$ and hence the problem is singular.

\begin{example}\rm\label{ex:simple_QEP_ZGV}
For a simple example we take
\[
L_2 = \smtxa{rr}{1 & 1 \\ 1 & 2},\quad
L_1 = \smtxa{rr}{0 & 3 \\ -3 & 0},\quad
L_0 = \smtxa{rr}{-2 & 1 \\ 1 & -2},\quad
M = \smtxa{rr}{3 & 1 \\ 1 & 4}.
\]
We select the matrices so that $L_0$ is symmetric, $L_1$ is skew-symmetric, and $L_2, M$ are symmetric positive definite. As a result, the matrices have the same properties as the larger matrices in \cite{ZGV_JASA_23}, where ZGV points of antisymmetric Lamb waves in an austenitic steel plate are computed.

The corresponding quadratic matrix polynomial \eqref{eq:ZGV_QEP} with $8\times 8$ matrices $\Gamma_2,\Gamma_1,\Gamma_0$ has normal rank $6$. The problem has a double eigenvalue at $0$
of geometric multiplicity 2 because $\rank(\Gamma_0)=4$. Algorithm~1 returns
the eigenvalues displayed in \ch{Table~\ref{tab:simple_QEP_ZGV}}.

\begin{table}[htb!]
\centering
\caption{Results of Algorithm~1 (rank-completing perturbation) applied to the singular QEP \eqref{eq:ZGV_QEP} with matrices from Example~\ref{ex:simple_QEP_ZGV}\label{tab:simple_QEP_ZGV}}
\smallskip
{\footnotesize \begin{tabular}{c|clllll} \hline \rule{0pt}{2.3ex}%
$j$ & $i k_j$ & $\quad \ \gamma_j$ & $\quad \ \alpha_j$ & $\quad \ \beta_j$ & ${\rm gap}_j$ & Type \\[0.5mm]
\hline \rule{0pt}{2.5ex}%
1 & $0$ & $6.7\cdot 10^{-2}$ & $1.6\cdot 10^{-15}$ & $2.6\cdot 10^{-15}$ & $0.00$ & Finite true \\
2 & $0$ & $6.3\cdot 10^{-2}$ & $1.8\cdot 10^{-15}$ & $1.6\cdot 10^{-15}$ & $0.00$ & Finite true \\
3 & $-4.007967\cdot 10^{-16} + 1.016018i$ & $6.1\cdot 10^{-2}$ & $1.6\cdot 10^{-14}$ & $1.3\cdot 10^{-14}$ & $0.37$ & Finite true \\
4 & $\ph-1.716439\cdot 10^{-15}- 1.016018i$ & $4.7\cdot 10^{-2}$ & $6.5\cdot 10^{-15}$ & $4.2\cdot 10^{-14}$ & $0.33$ & Finite true \\
5 & $-4.004034 - 7.895649\cdot 10^{-15}i$ & $5.3\cdot 10^{-3}$ & $1.2\cdot 10^{-14}$ & $5.3\cdot 10^{-14}$ & $0.61$ & Finite true \\
6 & $\ph-4.004034 - 1.035145\cdot 10^{-14}i$ & $4.4\cdot 10^{-3}$ & $1.8\cdot 10^{-14}$ & $2.5\cdot 10^{-14}$ & $0.87$ & Finite true \\
7 & $-\infty$ & $0.0$ & $4.2\cdot 10^{-15}$ & $4.8\cdot 10^{-15}$ & $1.00$ & Infinite true \\
8 & $\ph-\infty$ & $0.0$ & $4.8\cdot 10^{-15}$ & $2.6\cdot 10^{-15}$ & $1.00$ & Infinite true \\
9 & $-0.202213 -0.301941 i$ & $3.7\cdot 10^{-3}$ & $5.3\cdot 10^{-15}$ & $0.54$ & $0.18$ & Random right \\
10 & $-3.322434 +8.098093 i$ & $1.2\cdot 10^{-4}$ & $8.2\cdot 10^{-15}$ & $0.56$ & $0.72$ & Random right \\
11 & $\ph-0.114927 + 0.193407 i$ & $1.1\cdot 10^{-2}$ & $0.29$ & $8.6\cdot 10^{-15}$ & $0.11$ & Random left \\
12 & $\ph-9.990344 - 1.667635\cdot 10^{1} i$ & $4.7\cdot 10^{-5}$ & $0.23$ & $2.4\cdot 10^{-14}$ & $0.91$ & Random left \\
13 & $-0.815756 + 2.188693 i$ & $3.2\cdot 10^{-3}$ & $0.24$ & $0.72$ & $0.39$ & Prescribed \\
14 & $-0.815756 - 2.188693 i$ & $3.1\cdot 10^{-3}$ & $0.29$ & $0.60$ & $0.39$ & Prescribed \\
15 & $-1.552148 - 1.225251 i$ & $9.4\cdot 10^{-4}$ & $0.17$ & $0.55$ & $0.43$ & Prescribed \\
16 & $-1.552148 + 1.225251 i$ & $7.9\cdot 10^{-4}$ & $0.15$ & $0.52$ & $0.43$ & Prescribed \\
\hline
\end{tabular}}
\end{table}

Notice that the double eigenvalue $0$ is computed exactly due to a preprocessing step in {\tt quadeig} that reveals and deflates the zero and infinite eigenvalues contributed by singular leading and trailing matrix coefficients. Notice also that, although $0$ is a double eigenvalue, the corresponding reciprocals of the condition numbers $\gamma_1$ and $\gamma_2$ are far away from zero because the eigenvalue is semisimple.

There are four eigenvalues $ik$ on the imaginary axis
that give $k$-coordinates of ZGV points $(\omega,k)$. For each of these values of $k$ we solve the generalized eigenvalue
problem \eqref{eq:ZGV1} to compute $\omega$ and a corresponding eigenvector $u$.
Since all matrices are symmetric, $u$ is also a left eigenvector and the condition
for a ZGV point is
\[
u^*\,(2ikL_2+L_1)\,u=0,
\]
see \cite{ZGV_JASA_23}. We get four real ZGV points
$(1.110602, 0)$, $(0.470226, 0)$, $(0.364791,\pm 1.016018)$ which are shown together with real dispersion curves in Figure~\ref{fig:one}.

\begin{figure}[htb!]
 \centering
 \includegraphics[width=7cm]{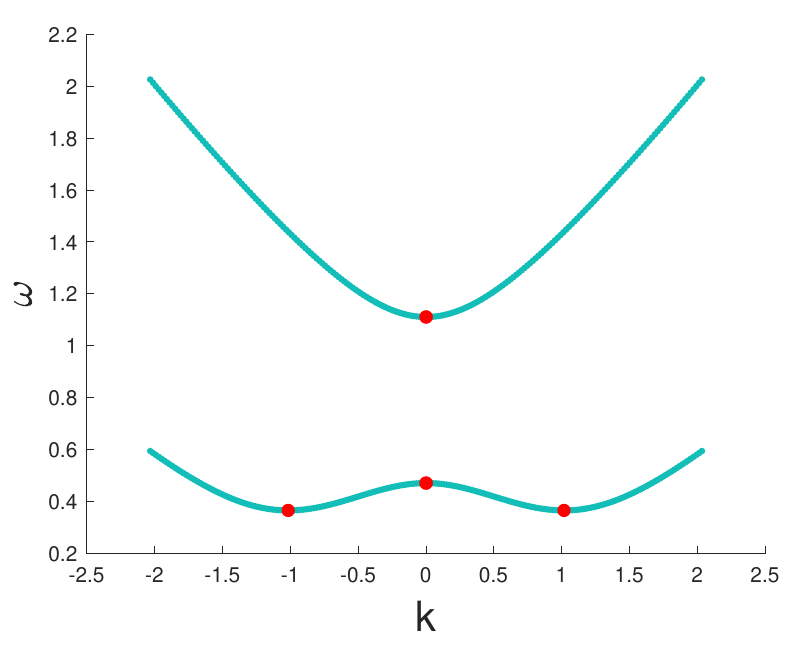}\label{fig:simple_QEP_ZGV}
 \caption{Real dispersion curves $\omega(k)$ and ZGV points of Example~\ref{ex:simple_QEP_ZGV}}
 \label{fig:one}
\end{figure}
\end{example}

As in the previous example, based on 10000 runs, we get $p_1=p_2=p_3=1$, ${\rm maxerr}_1=1.7\cdot 10^{-12}$,
${\rm maxerr}_2=1.2\cdot 10^{-11}$, and ${\rm maxerr}_2=7.0\cdot 10^{-12}$.

\subsection{Set of singular QEP examples}\label{sec:DanielIvana}
In a recent paper \cite{DanielIvana},
Kressner and \v{S}ain Glibi\'c construct several examples
of singular QEPs to test their numerical method, where they perturb a singular QEP into a regular problem by a small
random full rank perturbation, compute the eigenvalues
together with the left and right eigenvectors and then
identify the finite eigenvalues based on the size of
the condition number \eqref{def:kp}. This might be compared to
Algorithm~1 followed by Algorithm~4, but since a full rank perturbations are used
in \cite{DanielIvana}, the eigenvalues of the perturbed problem are perturbed as well.
We denote the algorithm from \cite{DanielIvana} as Algorithm K\v{S}G.

For each of the problems they perform 1000 runs of the algorithm
with different random perturbation matrices and computed empirical probabilities as the ratio of runs when the algorithm found all the finite eigenvalues. We test Algorithm K\v{S}G and Algorithms 1--3, followed by Algorithm~4, on the most challenging examples from \cite{DanielIvana} in a similar way. As another option we test standard linearizations from $\mathbb L_1(P)$ combined with the projection algorithm for a singular generalized eigenvalue problem from \cite{HMP_SingGep2}, this option is denoted ``Linearization''.

\begin{example}\label{ex:DanielIvana}\rm Our selected QEPs are (the first four problems are \cite[Exs.~5--8]{DanielIvana}, while 5) is new):
\begin{enumerate}
\item[1)] $Q(\lambda)=Z^*(K+\lambda C+\lambda^2M)\,W$ with $8\times 8$ matrices
and finite eigenvalues $\lambda_j=1+10^{-5}j$ for $j=1, \dotsc,5$. The nonzero elements of $K,C,M$ are
$m_{j,j+1}=c_{jj}=1$ and
$c_{j,j+1}=k_{jj}=-\lambda_j$ for $j=1, \dotsc,5$,
{\sf W = orth(rand(8))} and {\sf Z = orth(rand(8))}.
\item[2)] $Q(\lambda)=Z^*(K+\lambda C+\lambda^2M)\,W$ with $11\times 11$ matrices
and finite eigenvalues $\lambda_1=0$, $\lambda_j=1/j$ for $j=2, \dotsc,8$. The nonzero elements of $K,C,M$ are
$m_{j,j+1}=c_{jj}=1$ and
$c_{j,j+1}=k_{jj}=-\lambda_j$ for $j=1, \dotsc,8$,
{\sf W = orth(rand(11))} and {\sf Z = orth(rand(11))}.
\item[3)] The reversal of 2), i.e.,
$Q(\lambda)=Z^*(M + \lambda C + \lambda^2K)\,W$.
Finite eigenvalues are $\lambda_j=j+1$ for $j=1, \dotsc,7$.
\item[4)] $Q(\lambda)=Z^*(DMD + \lambda DCD + \lambda^2DKD)\,W$, where $D$ is the scaling matrix $D=\diag(1,a^2,a,1,a^3,1,a^4,a^5,a^6,1,1)$ for $a=2,4,6,8$. The other matrices and
finite eigenvalues are the same as in 3). The diagonal scaling makes the eigenvalues more ill conditioned and more difficult to detect and compute accurately.
\item[5)] $Q(\lambda)=Z^*(D^{-1}KD + \lambda D^{-1}CD + \lambda^2D^{-1}MD)\,W$ with $8\times 8$ matrices $K,C,M$,
whose nonzero elements are
$m_{j,j+1}=c_{jj}=1$, $m_{13}=m_{24}=1$, and
$c_{j,j+1}=k_{jj}=-1$ for $j=1, \dotsc,5$, and
$D=\diag(1,a^3,a^6,a^2,a^5,a,a^4,a^7)$ for $a=1,2,3$. The problem has finite eigenvalue $1$ with algebraic multiplicity $4$ and geometric multiplicity $3$.
\end{enumerate}

For each problem we carry out $N=10000$ runs of each algorithm and counted the number
of runs $F$ when the method fails to return the exact number of eigenvalues. The ratio $p=1-F/N$ is the empirical probability that the algorithm returns the correct answer. We also compute the maximal absolute error of the computed eigenvalues in all successful attempts, we report this as the maximal error. The results are listed in \ch{Table~\ref{tab:examplesDK}}. For Algorithm K\v{S}G we used parameters $\eps=10^{-12}$
for the size of the perturbation and ${\sf tol}=10^{10}$
for the tolerance. With these settings we obtained better
empirical probabilities and smaller error than reported in \cite{DanielIvana}.
For Linearization we used the threshold parameter $\delta=10^{-12}$.

\begin{table}[htb!]
\centering
\caption{Results of Algorithm K\v{S}G \cite[Alg.~2]{DanielIvana}, Algorithm~1 (rank-completing perturbation), Algorithm~2 (projection to normal rank), Algorithm~3 (augmentation), and Linearization (\cite[Alg.~2]{HMP_SingGep2} applied to an $\mathbb L_1$ linearization) applied to the singular QEPs from Example~\ref{ex:DanielIvana}.\label{tab:examplesDK}}
\smallskip
{\footnotesize \begin{tabular}{l|cl|cl|cl|cl|cl} \hline \rule{0pt}{2.3ex}%
 & \multicolumn{2}{c|}{Algorithm K\v{S}G } &
 \multicolumn{2}{c|}{Algorithm~1} & \multicolumn{2}{c|}{Algorithm~2} & \multicolumn{2}{c|}{Algorithm~3} & \multicolumn{2}{c}{Linearization} \\
 Problem & $F$ & max error & $F$ & max error & $F$ & max error & $F$ & max error & $F$ & max error \\[0.5mm]
\hline \rule{0pt}{2.5ex}%
 1) & $\ph{123}0$ & $1.5\cdot 10^{-10}$ & $0$ & $6.3\cdot 10^{-13}$ & $0$ & $1.5\cdot 10^{-13}$ & $0$ & $2.5\cdot 10^{-13}$ & $\ph00$ & $8.6\cdot 10^{-13}$ \\
 2) & $\ph{123}0$ & $4.8\cdot 10^{-11}$ & $0$ & $1.7\cdot 10^{-13}$ & $0$ & $3.4\cdot 10^{-14}$ & $0$ & $2.0\cdot 10^{-13}$ & $\ph00$ & $4.8\cdot 10^{-13}$\\
 3) & $\ph{123}0$ & $4.8\cdot 10^{-9}$ & $0$ & $4.8\cdot 10^{-12}$ & $0$ & $8.8\cdot 10^{-13}$ & $0$ & $1.6\cdot 10^{-11}$ & $\ph00$ & $6.3\cdot 10^{-12}$\\
 4) $(a=2)$ & $\ph{123}1$ & $2.5\cdot 10^{-7}$ & $0$ & $9.0\cdot 10^{-11}$ & $0$ & $6.7\cdot 10^{-11}$ & $0$ & $2.2\cdot 10^{-10}$ & $\ph00$ & $5.1\cdot 10^{-10}$\\
 4) $(a=4)$ & $\ph{123}3$ & $1.6\cdot 10^{-4}$ & $0$ & $6.6\cdot 10^{-8}$ & $0$ & $1.4\cdot 10^{-7}$ & $0$ & $2.0\cdot 10^{-7}$ & $\ph00$ & $6.1\cdot 10^{-7}$ \\
 4) $(a=6)$ & $\ph{12}68$ & $2.6\cdot 10^{-3}$ & $0$ & $9.3\cdot 10^{-6}$ & $0$ & $1.0\cdot 10^{-5}$ & $0$ & $2.1\cdot 10^{-5}$ & $\ph00$ & $1.7\cdot 10^{-5}$\\
 4) $(a=8)$ & $1204$ & $1.6\cdot 10^{-2}$ & $3$ & $1.9\cdot 10^{-4}$ & $1$ & $1.5\cdot 10^{-4}$ & $3$ & $3.4\cdot 10^{-4}$ & $20$ & $6.5\cdot 10^{-4}$\\
 5) $(a=1)$ & $\ph{123}0$ & $5.7\cdot 10^{-6}$ & $0$ & $2.8\cdot 10^{-7}$ & $0$ & $3.9\cdot 10^{-7}$ & $0$ & $3.2\cdot 10^{-7}$ & $\ph00$ & $4.4\cdot 10^{-7}$\\
 5) $(a=2)$ & $\ph{12}50$ & $5.1\cdot 10^{-5}$ & $0$ & $2.9\cdot 10^{-6}$ & $0$ & $2.5\cdot 10^{-6}$ & $0$ & $5.7\cdot 10^{-6}$ & $\ph00$ & $7.2\cdot 10^{-6}$\\
 5) $(a=3)$ & $3020$ & $1.4\cdot 10^{-4}$ & $0$ & $2.4\cdot 10^{-5}$ & $0$ & $1.4\cdot 10^{-5}$ & $0$ & $3.5\cdot 10^{-5}$ & $\ph00$ & $8.1\cdot 10^{-5}$\\
 \hline
\end{tabular}}
\end{table}

The singular QEP in 5) has a quadruple eigenvalue $\lambda=1$. \ch{Table \ref{tab:35}} shows the output of Algorithm~2 for $a=3$.
We can see that the values $\gamma_i$ for the quadruple eigenvalue are close to $0$ as are the relative gaps, as expected \ch{in the discussion just before}
Algorithm~4. \ch{In particular, for $j=3$ and $j=4$ we see that $\gamma_j>\delta_1=10^{-16}$ and ${\rm gap}_j<\xi_2=10^{-2}$, therefore these two very ill-conditioned eigenvalues are identified as finite eigenvalues by Algorithm~4.}

\begin{table}[htb!]
\centering
\caption{Results of Algorithm~2 (projection to normal rank) applied to the singular
QEP 5) from Example~\ref{ex:DanielIvana} for $a=3$.\label{tab:35}}
\smallskip
{\footnotesize \begin{tabular}{c|llllll} \hline \rule{0pt}{2.3ex}%
$j$ & \hspace{17mm} $\lambda_j$ & $\quad \ \gamma_j$ & $\quad \ \alpha_j$ & $\quad \ \beta_j$ & ${\rm gap}_j$ & Type \\[0.5mm]
\hline \rule{0pt}{2.5ex}%
1 & $\ph-1.000000 + 1.036016\cdot 10^{-15}i$ & $1.1\cdot 10^{-4}$ & $1.4\cdot 10^{-16}$ & $7.7\cdot 10^{-17}$ &
$2.1\cdot 10^{-11}$ & Finite true \\
2 & $\ph-1.000000 + 1.549034\cdot 10^{-11}i$ & $1.3\cdot 10^{-8}$ & $1.8\cdot 10^{-16}$ & $6.6\cdot 10^{-17}$ &
$2.1\cdot 10^{-11}$ & Finite true \\
3 & $\ph-0.999999 - 6.066428\cdot 10^{-7}i$ & $7.2\cdot 10^{-13}$ & $1.4\cdot 10^{-16}$ & $4.0\cdot 10^{-17}$ &
$6.4\cdot 10^{-7}$ & Finite true \\
4 & $\ph-1.000000 + 6.066328\cdot 10^{-7}i$ & $7.2\cdot 10^{-13}$ & $1.3\cdot 10^{-16}$ & $4.3\cdot 10^{-17}$ &
$6.4\cdot 10^{-7}$ & Finite true \\
5 & $-0.748612 + 7.673687\cdot 10^{-2}i$ & $2.0\cdot 10^{-7}$ & $1.0\cdot 10^{-16}$ & $9.0\cdot 10^{-6}$ & $1.18$ & Random right\\
6 & $-0.355925 - 0.716723i$ & $1.0\cdot 10^{-7}$ & $1.6\cdot 10^{-16}$ & $5.2\cdot 10^{-6}$ & $1.11$ & Random right\\
7 & $ \ph-0.591362 + 0.242587i$ & $7.6\cdot 10^{-8}$ & $1.3\cdot 10^{-16}$ & $2.4\cdot 10^{-6}$ & $0.74$ & Random right\\
8 & $-0.060825 + 0.828186i$ & $6.9\cdot 10^{-8}$ & $1.2\cdot 10^{-16}$ & $3.9\cdot 10^{-6}$ & $1.05$ & Random right\\
9 & $\ph-0.579022 - 0.389633i$ & $6.2\cdot 10^{-8}$ & $1.3\cdot 10^{-16}$ & $2.5\cdot 10^{-6}$ & $0.82$ & Random right\\
10 & $\ph-3.012867 - 1.129741\cdot 10^{-2}i$ & $2.6\cdot 10^{-8}$ & $1.5\cdot 10^{-16}$ & $1.3\cdot 10^{-7}$ & $0.67$ & Random right\\
\hline
\end{tabular}}
\end{table}

\end{example}

\subsection{Other applications}

Although there are several references to applications of singular matrix polynomials in the literature, we managed to collect just few concrete examples.

\begin{example}\rm\label{ex:kk}
In \cite{KafetzisKarampetakis_MPinverse}, where a generalization of the Moore--Penrose inverse has been introduced for singular matrix polynomials, the following singular polynomial matrix of degree $5$ is mentioned \cite[Ex.~1] {KafetzisKarampetakis_MPinverse}:
\begin{equation}\label{eq:kk}
P(\lambda)=
\smtxa{ccc}{
1 + 4 \lambda + 5\lambda^2 + 2\lambda^3 & -1 -3\lambda -4\lambda^2 -3\lambda^3-\lambda^4 & -\lambda-2\lambda^2-\lambda^3 \\[0.5mm]
-1 -2\lambda + 2\lambda^2 + 5\lambda^3+ 2\lambda^4 & 1 + \lambda - \lambda^2 -3\lambda^3-3\lambda^4-\lambda^5 & \lambda -2\lambda^3-\lambda^4 \\[0.5mm]
-1 -2\lambda + \lambda^2+2\lambda^3 & 1 +\lambda - \lambda^3-\lambda^4 & \lambda - \lambda^3},
\end{equation}
where $\nrank(P(\lambda))=1$. The problem has one finite eigenvalue $\lambda=-1$, which is identified correctly
by Algorithm~2; results of a typical run are presented in \ch{Table~\ref{tab:kk}}.

\begin{table}[htb!]
\centering
\caption{Results of Algorithm~2 (projection to normal rank) applied to the singular PEP \eqref{eq:kk}.\label{tab:kk}}
\smallskip
{\footnotesize \begin{tabular}{c|llllll} \hline \rule{0pt}{2.3ex}%
$j$ & \hspace{17mm} $\lambda_j$ & $\quad \ \gamma_j$ & $\quad \ \alpha_j$ & $\quad \ \beta_j$ & ${\rm gap}_j$ & Type \\[0.5mm]
\hline \rule{0pt}{2.5ex}%
1 & $-1.000000 + 1.016777\cdot 10^{-16}i$ & $6.4\cdot 10^{-2}$ & $1.3\cdot 10^{-16}$ & $1.4\cdot 10^{-16}$ & $0.56$ & Finite true \\
2 & $-0.734887 -0.835476i$ & $8.8\cdot 10^{-2}$ & $1.2\cdot 10^{-16}$ & $5.1\cdot 10^{-2}$ & $0.54$ & Random right \\
3 & $-1.108317 + 1.038442i$ & $5.5\cdot 10^{-2}$ & $1.1\cdot 10^{-16}$ & $5.1\cdot 10^{-2}$ & $0.48$ & Random right \\
4 & $\ph-0.230639 - 1.285601i$ & $2.1\cdot 10^{-1}$ & $2.5\cdot 10^{-1}$ & $2.7\cdot 10^{-16}$ & $0.60$ & Random left\\
5 & $-0.190740 + 0.980385i$ & $1.6\cdot 10^{-1}$ & $1.7\cdot 10^{-1}$ & $9.2\cdot 10^{-17}$ & $0.59$ & Random left\\
\hline
\end{tabular}}
\end{table}

Note that in this case the projected polynomial is of size $1\times 1$, thus we need to find roots of a scalar polynomial of degree $5$. The corresponding left and right eigenvectors from Algorithm~2 are scalars
$x_i=y_i=1$. Still, together with the initial unitary transformations $U$ and $V$ they provide enough
information that we can classify the eigenvalues. If we instead first linearize $P(\lambda)$ into a linear pencil,
then we get a singular pencil with matrices $15\times 15$ that has normal rank $13$. This example shows that it may be considerably more efficient to apply Algorithm~2,
in particular for polynomials of high degree.

We perform $N=10000$ runs and obtained empirical probabilities $p_1=0.9983$, $p_2=1$ and $p_3=0.9968$ for Algorithms~1--3 with maximal errors ${\rm maxerr}_1=7.1\cdot 10^{-13}$, ${\rm maxerr}_2=4.2\cdot 10^{-13}$, and
${\rm maxerr}_3=4.9\cdot 10^{-13}$. By increasing the threshold $\delta$ to $10^{-8}$ for Algorithm~1 and $10^{-10}$ for Algorithm~3
we get $p_1=p_3=1$.

\end{example}

\begin{example}\rm\label{ex:zh}
In \cite[Ex.~7]{ZunigaHenrion} we find the following singular polynomial matrix
of degree $8$:
\begin{equation}\label{eq:zh}
P(\lambda)=
\smtxa{ccc}{
\lambda^2+ \lambda^8 & \lambda+\lambda^7 & \ph-\lambda^4 \\[0.5mm]
-\lambda -\lambda^7 & -1 - \lambda^6 & -\lambda^3\\[0.5mm]
\lambda^4 & \lambda^3 & \ph-1},
\end{equation}
where $\nrank(P(\lambda))=2$. The problem has 14 infinite eigenvalues but no finite eigenvalues.
The results
obtained by Algorithm~2 are presented in \ch{Table~\ref{tab:zh}}. The values in rows 2--11 that we left out are similar to the values in rows 1 and 12. We get two exact infinite eigenvalues in rows 13--14 because the leading matrix coefficient has rank 1, while
the remaining infinite eigenvalues are numerically evaluated as simple eigenvalues
$\lambda_1,\ldots,\lambda_{12}$ with absolute values ${\cal O}(10^1)$ that are correctly identified as infinite eigenvalues by Algorithm~4.
The empirical probabilities that the methods do not
return any finite eigenvalues are $p_1=p_2=p_3=1$ for Algorithms~1--3.

\begin{table}[htb!]
\centering
\caption{Results of Algorithm~2 (projection to normal rank) applied to the singular PEP \eqref{eq:zh}.\label{tab:zh}}
\smallskip
{\footnotesize \begin{tabular}{c|clllll} \hline \rule{0pt}{2.3ex}%
$j$ & $\lambda_j$ & $\quad \ \gamma_j$ & $\quad \ \alpha_j$ & $\quad \ \beta_j$ & ${\rm gap}_j$ & Type \\[0.5mm]
\hline \rule{0pt}{2.5ex}%
1 & $-3.01396 + 17.32212i$ & $1.1\cdot 10^{-17}$ & $1.7\cdot 10^{-16}$ & $2.1\cdot 10^{-16}$ & $0.51$ & Infinite true \\[-1mm]
$\vdots$ & $\vdots$ & \multicolumn{1}{c}{$\vdots$} & \multicolumn{1}{c}{$\vdots$} & \multicolumn{1}{c}{$\vdots$} & \multicolumn{1}{c}{$\vdots$} & \multicolumn{1}{c}{$\vdots$}\\
12 & $17.16068 +3.09912i$ & $9.0\cdot 10^{-19}$ & $1.3\cdot 10^{-16}$ & $2.1\cdot 10^{-16}$ & $0.52$ & Infinite true\\
13 & $\infty$ & $0.0$ & $1.7\cdot 10^{-16}$ & $6.8\cdot 10^{-16}$ & $1.00$ & Infinite true\\
14 & $\infty$ & $0.0$ & $1.3\cdot 10^{-15}$ & $2.6\cdot 10^{-16}$ & $1.00$ & Infinite true\\
15 & $0.364834 - 0.699938i$ & $1.5\cdot 10^{-1}$ & $4.3\cdot 10^{-16}$ & $1.2\cdot 10^{-1}$ & $0.65$ & Random right \\
16 & $0.258964 - 1.516862i$ & $2.2\cdot 10^{-3}$ & $3.0\cdot 10^{-3}$ & $3.3\cdot 10^{-16}$ & $0.45$ & Random right \\
\hline
\end{tabular}}
\end{table}

\end{example}

\ch{
\subsection{Incorrectly determined rank}

All presented methods rely on the normal rank. If we fail to
determine the normal rank correctly, then we observe a
behavior similar to the one reported in \cite[Sec.~8]{HMP_SingGep2} for singular pencils.

\begin{example}\rm\label{ex:new1}
In this example we show the effects of using an incorrect normal rank on the singular QEP 5) from Example~\ref{ex:DanielIvana} for $a=3$, whose
normal rank is 5 and has a quadruple eigenvalue $1$ of geometric multiplicity three. We apply Algorithm~2 with projection to normal rank and
use incorrect values 4 and 6 for the normal rank. The obtained results should be compared to values
in Table~\ref{tab:35}, where the correct normal rank has been applied.

The results for the underestimated normal rank $4$ are presented in Table~\ref{tab:35under}.
In this case only multiple eigenvalues with geometric multiplicities large enough
are detected as true eigenvalues and we get two out of four instances of the eigenvalue $1$, while all other eigenvalues are
identified as prescribed eigenvalues.
Since normally in Algorithm 2 we do not expect any prescribed eigenvalues, this
is a clear sign that \cm{the normal rank was computed incorrectly}.

\begin{table}[htb!]
\centering
\caption{Results of Algorithm~2 (projection to normal rank) applied to the singular
QEP 5) from Example~\ref{ex:DanielIvana} for $a=3$, using an underestimated value 4 for the normal rank.\label{tab:35under}}
\smallskip
{\footnotesize \begin{tabular}{c|clllll} \hline \rule{0pt}{2.3ex}%
$j$ & $\lambda_j$ & $\quad \ \gamma_j$ & $\quad \ \alpha_j$ & $\quad \ \beta_j$ & ${\rm gap}_j$ & Type \\[0.5mm]
\hline \rule{0pt}{2.5ex}%
1 & $\ph-1.000000 - 1.098315\cdot 10^{-15}i$ & $1.6\cdot 10^{-4}$ & $9.7\cdot 10^{-17}$ & $7.1\cdot 10^{-17}$ &
$8.9\cdot 10^{-12}$ & Finite true \\
2 & $\ph-1.000000 - 5.754158\cdot 10^{-13}i$ & $1.6\cdot 10^{-7}$ & $9.4\cdot 10^{-17}$ & $9.3\cdot 10^{-17}$ &
$8.9\cdot 10^{-12}$ & Finite true \\
3 & $-3.226097\cdot 10^{-2} - 2.206336\cdot 10^{-2}i$ & $1.9\cdot 10^{-1}$ & $2.5\cdot 10^{-1}$ & $1.4\cdot 10^{-1}$ &
$1.59$ & Prescribed \\
4 & $\ph-3.996238\cdot 10^{-2} + 2.377873\cdot 10^{-2}i$ & $1.1\cdot 10^{-1}$ & $1.6\cdot 10^{-1}$ & $9.5\cdot 10^{-2}$ &
$1.44$ & Prescribed \\[-1mm]
$\vdots$ & $\vdots$ & \multicolumn{1}{c}{$\vdots$} & \multicolumn{1}{c}{$\vdots$} & \multicolumn{1}{c}{$\vdots$} & \hspace{2mm} $\vdots$ & \multicolumn{1}{c}{$\vdots$}\\
8 & $\ph-2.901284 - 0.916448i$ & $8.5\cdot 10^{-8}$ & $1.2\cdot 10^{-5}$ & $1.6\cdot 10^{-5}$ & $0.69$ & Prescribed\\
\hline
\end{tabular}}
\end{table}

The results for the overestimated normal rank $6$ are presented in Table~\ref{tab:35over}.
The transformed problem is still singular because we project matrices to a size larger than the correct normal rank $5$. Now all
values $\alpha_j$ and $\beta_j$ are close to zero and all eigenvalues are identified as true eigenvalues. A
subset of computed eigenvalues is close to true finite eigenvalues and we can identify them in the second phase by Algorithm 4.
If we are sure that the initial PEP is singular, then the fact that all eigenvalues are identified as true eigenvalues, is a sign that
the normal rank is overestimated.

\begin{table}[htb!]
\centering
\caption{Results of Algorithm~2 (projection to normal rank) applied to the singular
QEP 5) from Example~\ref{ex:DanielIvana} for $a=3$, using an overestimated value 6 for the normal rank.\label{tab:35over}}
\smallskip
{\footnotesize \begin{tabular}{c|clllll} \hline \rule{0pt}{2.3ex}%
$j$ & $\lambda_j$ & $\quad \ \gamma_j$ & $\quad \ \alpha_j$ & $\quad \ \beta_j$ & ${\rm gap}_j$ & Type \\[0.5mm]
\hline \rule{0pt}{2.5ex}%
1 & $1.000000 + 4.385491\cdot 10^{-15}i$ & $2.2\cdot 10^{-5}$ & $1.6\cdot 10^{-16}$ & $3.0\cdot 10^{-17}$ &
$2.9\cdot 10^{-12}$ & Finite true \\
2 & $1.000000 + 1.649253\cdot 10^{-12}i$ & $5.6\cdot 10^{-8}$ & $1.7\cdot 10^{-16}$ & $3.1\cdot 10^{-17}$ &
$2.9\cdot 10^{-12}$ & Finite true \\
3 & $1.000000 + 1.782704\cdot 10^{-6}i$ & $3.8\cdot 10^{-14}$ & $1.4\cdot 10^{-16}$ & $3.1\cdot 10^{-17}$ &
$1.8\cdot 10^{-6}$ & Finite true \\
4 & $1.000000 - 1.782686\cdot 10^{-6}i$ & $3.8\cdot 10^{-14}$ & $1.3\cdot 10^{-16}$ & $3.9\cdot 10^{-17}$ &
$1.8\cdot 10^{-6}$ & Finite true \\
5 & $0$ & $3.6\cdot 10^{-17}$ & $2.3\cdot 10^{-16}$ & $1.8\cdot 10^{-16}$ & $3.17$ & Infinite true\\
6 & $0.116378 - 1.424291\cdot 10^{-2}i$ & $9.1\cdot 10^{-18}$ & $7.7\cdot 10^{-17}$ & $4.5\cdot 10^{-17}$ & $0.95$ & Infinite true \\[-1mm]
$\vdots$ & $\vdots$ & \multicolumn{1}{c}{$\vdots$} & \multicolumn{1}{c}{$\vdots$} & \multicolumn{1}{c}{$\vdots$} & \hspace{2mm} $\vdots$ & \multicolumn{1}{c}{$\vdots$}\\
11 & $2.999570 + 0.257650i$ & $3.9\cdot 10^{-19}$ & $1.5\cdot 10^{-16}$ & $7.5\cdot 10^{-17}$ & $0.67$ & Infinite true\\
12 & $\infty$ & $0.0$ & $6.1\cdot 10^{-17}$ & $4.1\cdot 10^{-17}$ & $1.00$ & Infinite true\\
\hline
\end{tabular}}
\end{table}
\end{example}
}

\ch{
\subsection{Rectangular matrix polynomials}
The methods can be applied to rectangular singular PEPs as well. For Algorithms 1 and 3 the solution is
to add zero rows or columns to make \cm{the underlying} matrices square, which does not change the regular eigenvalues. This is not
required for Algorithm 2, where we simply project the matrices to square matrices of dimension of the normal rank. If we
consider the computational costs, then Algorithm 2 is clearly the most effective approach, in particular if the matrices
are very skinny or fat.

As the principle is the same as for rectangular singular pencils, we do not give a particular example but refer to
\cite[Ex.~6.2]{HMP_SingGep} and
\cite[Ex.~7.6]{HMP_SingGep2}.}

\ch{
\subsection{Computational complexity}

The main computational task of Algorithms 1-3 is to solve the transformed regular PEP. If \cm{the underlying} matrices
are full and we use a standard solver such as, e.g., \texttt{polyeig} in Matlab, then all additional computations, including the
estimation of the normal rank and the computation of the projected matrices in Algorithm~2, present just a negligible part.
For large $n$, Algorithm~1 is just marginally more expensive than \texttt{polyeig}.
Algorithms 2 and 3 are expected to be faster and slower, respectively, due to smaller and larger matrices.
We remark that for QEPs Algorithm K\v{S}G from \cite{DanielIvana} is asymptotically two times as expensive as Algorithm~1, because it computes the eigenvalues of the perturbed regular QEP using two different linearizations.}

\section{Conclusions}\label{sec:conclusion}
We have studied singular polynomial eigenvalue problems, and extended results from singular pencils \cite{HMP_SingGep, HMP_SingGep2} to this new situation.

In the PEP context, we exploit the Smith form as in Theorem~\ref{the:smith} instead of the Kronecker normal form for pencils.
Since there is no Kronecker canonical form for matrix polynomials under strict equivalence, we need different proof arguments.
In addition, one of the main technical challenges of this paper is to show that the so-called random eigenvalues are generically simple.
While there is a relatively simple and elegant argument for this in the pencil case (see \cite[p.~1034]{HMP_SingGep}), this is much more complicated for the PEP; cf.~the proof in the appendix.

As applications, we have considered certain bivariate polynomial systems, ZGV points, and several scattered examples from \cite{DanielIvana} and other references.
We compared some of our results with those of \cite{DanielIvana}.
An advantage of \cite{DanielIvana} is that no normal rank computation is needed; on
the other hand, the methods introduced in this paper appear to be a bit more accurate.
While all approaches generally perform very well, especially the projection approach, which is often also the most efficient one, seems to be slightly more reliable and accurate.

\section*{Acknowledgments}
We thank both referees for very useful and insightful suggestions.
BP has been supported by the Slovenian Research and Innovation Agency (grants N1-0154, P1-0294).
Parts of the results have been obtained
while the first two authors visited the third author at the University of Ljubljana.
MH and CM thank BP and the Faculty of Mathematics and Physics for their hospitality.

\appendix
% Redefine theorem numbering for appendix
\renewcommand{\thealem}{A.\arabic{alem}}
\renewcommand{\theathm}{A.\arabic{athm}}

\section{Simple random eigenvalues}

Let $P(\lambda)$ be an $n\times n$ singular matrix polynomial of degree
$d$ and normal rank $n-k$, and
let $S_R(\lambda)=\left[x_1(\lambda)\ \dots\ x_{k}(\lambda)\right]$
be an $n\times k$ matrix polynomial, such that its columns form a minimal basis
of the right nullspace of
$P(\lambda)$
with right minimal indices $m_1, \dotsc,m_k$. In the proof of Lemma \ref{lem:Lo} it was
already shown that generically with respect to the entries of
$V^*\in\mathbb C^{k\times n}$ the polynomial $\det\big(V^*S_R(\lambda)\big)$
has exactly $M=m_1+\cdots+m_k$ roots. In this appendix we complete the proof that
generically these roots are simple by providing one particular example for $V$
such that the roots of $\det\big(V^*S_R(\lambda)\big)$ are simple.

Recall that if $p_1(\lambda)$ and $p_2(\lambda)$ are two scalar polynomials, then their \emph{greatest common divisor} ${\rm gcd}\big(p_1(\lambda),p_2(\lambda)\big)$ is
a polynomial of maximal degree that divides both $p_1(\lambda)$ and $p_2(\lambda)$.
For more than two polynomials, \[{\rm gcd}\big(p_1(\lambda), \dotsc, p_k(\lambda)\big)=
 {\rm gcd}\big({\rm gcd}(p_1(\lambda), \dotsc, p_{k-1}(\lambda))\,,\,p_k(\lambda)\big).\]
It is well known that the greatest common divisor is unique up to multiplication with
a nonzero constant.

\begin{alem}\label{lem:newk}
 Let $p_1(\lambda), \dotsc, p_k(\lambda)$ be such that their greatest common divisor
 $q(\lambda)={\rm gcd}\big(p_1(\lambda), \dotsc, p_k(\lambda)\big)$ has simple roots.
 Then there exists a generic set $\Omega\subseteq\mathbb C^k$ such that for all
 $s=(s_1, \dotsc,s_k)\in\Omega$ the roots of the polynomial
 $s_1\,p_1(\lambda)+\cdots+s_k\,p_k(\lambda)$ are simple.
\end{alem}
\begin{proof}
Let $p_i(\lambda)=q(\lambda)\,h_i(\lambda)$ for $i=1,\ldots,k$, where
${\rm gcd}\big(h_1(\lambda), \dotsc, h_k(\lambda)\big)=1$.
Let $\eta_1, \dotsc,\eta_\ell\in\mathbb C$, where $\ell$ is the degree of $q(\lambda)$,
be the roots of $q(\lambda)$, that are simple by assumption.

If $h_1(\lambda),\ldots,h_k(\lambda)$ are all constant polynomials, then
$p_1(\lambda),\ldots,p_k(\lambda)$ are all multiples of $q(\lambda)$ and hence for any
$s\in\mathbb C^k$
such that $s_1h_1(\lambda)+\cdots+s_kh_k(\lambda)\ne 0$
the roots of $s_1p_1(\lambda)+\cdots+s_kp_k(\lambda)$
are $\eta_1, \dotsc,\eta_\ell$ and thus simple.

Assume now that the degree of at least one of the polynomials $h_1(\lambda),\ldots,h_k(\lambda)$ is nonzero. Suppose that $\xi$ is a multiple zero of $s_1h_1(\lambda)+\cdots + s_kh_k(\lambda)$ for
some $s \in \CC^k$. Then
$s_1h_1(\xi)+\cdots+ s_kh_k(\xi)=0$ and $s_1h_1'(\xi)+\cdots+s_k h_k'(\xi)=0$. It follows that $s$ is in the kernel of the $2\times k$ polynomial matrix
\begin{equation}\label{eq:mult_root_criteria2}
 H(\xi):=\left[\begin{matrix} h_1(\xi) & \cdots & h_k(\xi)\cr
 h_1'(\xi) & \cdots & h_k'(\xi)\end{matrix}\right]
\end{equation}
that does not depend on $s$. We know that
$\big(h_1(\xi),\ldots,h_k(\xi)\big)\ne (0,\ldots,0)$ for each $\xi\in\CC$ because
${\rm gcd}\big(h_1(\lambda), \dotsc, h_k(\lambda)\big)=1$, so
$\rank\big(H(\xi)\big)\ge 1$ for all $\xi$.

Let us show that $\nrank\big(H(\xi)\big)=2$. Without loss of generality we can assume
that degree of $h_1(\lambda)$ is at least one, so there exists $\zeta\in\CC$ such that
$h_1(\zeta)=0$. Since ${\rm gcd}\big(h_1(\lambda), \dotsc, h_k(\lambda)\big)=1$, at least
one of the remaining polynomials, which we can assume to be $h_2(\lambda)$, is
nonzero at $\zeta$. Now we consider the following polynomial, which is a minor of
$H(\lambda)$,
\begin{equation}\label{lem:nov_dokaz}
h_1(\lambda)\,h_2'(\lambda)-h_1(\lambda)\,h_2'(\lambda).
\end{equation}
Let $h_i(\lambda)=g(\lambda)r_i(\lambda)$ for $i=1,2$, where
$g(\lambda)={\rm gcd}\big(h_1(\lambda),h_2(\lambda)\big)$ and
${\rm gcd}\big(r_1(\lambda),r_2(\lambda)\big)=1$. We know that
$r_1(\lambda)$ is not a constant polynomial and that $r_1(\zeta)=0$.
Suppose now that \eqref{lem:nov_dokaz} is identically zero. Then
\[
g(\lambda)\,r_1(\lambda)\,\big(g(\lambda)\,r_2'(\lambda)+g'(\lambda)\,r_2(\lambda)\big) = g(\lambda)\,r_2(\lambda)\,\big(g(\lambda)\,r_1'(\lambda)+g'(\lambda)\,r_1(\lambda)\big)
\]
and it follows that $r_1(\lambda)\,r_2'(\lambda)=r_1'(\lambda)\,r_2(\lambda)$. But then
it follows from ${\rm gcd}\big(r_1(\lambda),r_2(\lambda)\big)=1$ that
$r_1(\lambda)$ divides its derivative $r_1'(\lambda)$, which is not possible.
Therefore \eqref{lem:nov_dokaz} is not identically zero and,
since this is a minor of $H(\lambda)$, this shows that $\nrank\big(H(\lambda)\big)=2$.

Then there are only finitely many
values $\xi_1,\ldots,\xi_r$ such that $\rank\big(H(\xi_i)\big)=1$ for $i=1,\ldots,r$. This is how we get $r+\ell$ hyperplanes $s_1h_1(\xi_i)+\cdots+ s_kh_k(\xi_i)=0$, $i=1,\ldots,r$, and
$s_1h_1(\eta_j)+\cdots+ s_kh_k(\eta_j)=0$, $j=1,\ldots,\ell$.
If $s$ is not on any of these hyperplanes, the roots of $s_1h_1(\lambda)+\cdots + s_kh_k(\lambda)$ are simple and different from $\eta_1,\ldots,\eta_\ell$.
The roots of $s_1p_1(\lambda)+\cdots + s_kp_k(\lambda)$ are then simple, and since
hyperplanes are algebraic sets, this holds generically for $s\in\CC^k$.\qed
\end{proof}

Let the columns of $S_R(\lambda)=\left[x_1(\lambda)\ \dots\ x_{k}(\lambda)\right]$
form a minimal basis of the right nullspace of an $n\times n$ singular
matrix polynomial
$P(\lambda)$ of degree $d$ and normal rank $n-k$ with minimal indices
$m_1, \dotsc,m_k$, where we can assume that $m_1\le \cdots \le m_k$.
Then $x_i(\lambda)$ is a polynomial of degree $m_i$ and we can write
$x_i(\lambda) = x_i^{(0)} + \lambda x_{i}^{(1)} + \cdots+
\lambda^{m_i}x_{i}^{(m_i)}$, where
$x_i^{(0)}, \dotsc, x_i^{(m_i)}\in\CC^n$ for $i=1, \dotsc,k$. Then from \cite{Forney_minimal_basis} we know the following:
\begin{enumerate}
 \item[1)] The \emph{high-order coefficient matrix} $S_h=[\,x_1^{(m_1)}\ \dots\ x_k^{(m_k)}\,]$ has full rank.
 \item[2)] $S_R(\lambda)$ has full rank for all $\lambda\in\CC$; in particular
 $S_R(0)= [\,x_1^{(0)}\ \dots\ x_k^{(0)}\,]$ has full rank.
 \item[3)] The greatest common divisor of all $k\times k$ minors of $S_R(\lambda)$ is $1$.
 \item[4)] We have ${\rm gcd}\big(x_{i1}(\lambda),\dots,x_{in}(\lambda)\big)=1$
 for all $i=1, \dotsc,k$.
\end{enumerate}

\begin{alem}\label{cnj:k3}
Let $S_R(\lambda)=[\,x_1(\lambda)\ \dots\ x_{k}(\lambda)\,]$
be an $n\times k$ matrix polynomial, such that its columns form a minimal basis
of the right nullspace of
$P(\lambda)$
with right minimal indices $m_1, \dotsc,m_k$, where $k\ge 2$.
Then, there exists vectors $v_1,\dots,v_{k-1}\in\mathbb C^{n}$ such that
\begin{equation}\label{eq:rank_def}
\rank\big([v_1\,\dots\,v_{k-1}]^*S_R(\lambda)\big)=k-1\ \textrm{for all }\ \lambda\in\CC.
\end{equation}
\end{alem}
\begin{proof}
We apply induction and start with $k=2$. Thus, we aim to find $v_1\in\mathbb C^n$
such that
$\operatorname{rank}\big([v_1^*x_1(\lambda) \ \ v_1^*x_2(\lambda)]\big) = 1$
for all $\lambda\in\mathbb C$.
Generically with respect to the entries of $v_1$ the polynomials $v_1^*x_1(\lambda)$
and $v_1^*x_2(\lambda)$ will have $m_1$ and $m_2$ \ch{simple} roots, respectively. Let us select
such a $v_1\in\mathbb C^n$ and let $\alpha_1,\dots,\alpha_{m_1}$ denote the roots of
$v_1^*x_1(\lambda)$ and $\beta_1,\dots,\beta_{m_2}$ those of $v_1^*x_2(\lambda)$.
(We highlight that enforcing the maximal possible degree of each polynomial
$v_{i}^*x_j(\lambda)$ ensures that the number of roots remains invariant under sufficiently small perturbations.)
Suppose that the roots are ordered in such a way that
\[
\alpha_i=\beta_i\quad\mbox{for }i=1,\dots,r\quad\mbox{ and }\quad
\ch{\alpha_j\neq\beta_\ell\quad\mbox{for }j=r+1,\dots,m_1,\ \ell=r+1,\dots,m_2.}
\]
In other words, we have exactly $r$ solutions $\lambda_i:=\alpha_i=\beta_i$,
$i=1,\dots,r$, to the equation
\begin{equation}\label{eq:induction2}
 v_1^*x_1(\lambda)=v_1^*x_2(\lambda)=0.
\end{equation}
We aim for $r=0$, because then we have
$\operatorname{rank}\big([v_1^*x_1(\lambda) \ \ v_1^*x_2(\lambda)]\big) = 1$
for all $\lambda\in\mathbb C$. Suppose this is not the case, i.e., $r>0$.
Since $x_1(\lambda_1)$ and $x_2(\lambda_1)$ are linearly independent, there
exists $w_1\in\CC^n$ such that $w_1^*x_1(\lambda_1)=0$ and
$w_1^*x_2(\lambda_1)\ne 0$. If we replace $v_1$ by $v_1+\delta_1 w_1$, then
for sufficiently small $\delta_1>0$ the polynomials $(v_1+\delta_1 w_1)^*x_j(\lambda)$ will still have $m_1$ zeros $\alpha_1,\dots,\alpha_{m_1}$ and $m_2$ zeros $\beta_1,\dots,\beta_{m_2}$, respectively, where we dropped
the dependency of the roots on $\delta_1$ for convenience. Note that we still have $\lambda_1=\alpha_1$ by construction. Furthermore,
we have $(v_1+\delta_1 w_1)^*[x_1(\lambda_1)\ \ x_2(\lambda_1)]=
[0\ \ \delta_1 w_1^*x_2(\lambda_1)]\ne[0\ \ 0]$ which means that
$\lambda_1$ is no longer a solution of \eqref{eq:induction2}.
Due to the continuity of zeros, we still have $\alpha_i\neq\beta_i$ for
$i=r+1,\dots,\min(m_1,m_2)$, so there can now be at most $r-1$ solutions of
\eqref{eq:induction2} if there still are any. We can remove those
in the same way one by one by further perturbing $v_1$. Then we have found
one particular vector for which~\eqref{eq:rank_def} holds in the case $k=2$.

Assume now that the statement holds for $k-1$. Then for
$\widehat S_R(\lambda)=[\,x_1(\lambda)\ \dots\ x_{k-1}(\lambda)\,]$
there exist $k-2$ vectors $v_1,\dots,v_{k-2}\in\mathbb C^n$ such that
\begin{equation}\label{eq:11.6.24}
\operatorname{rank}\big([v_1\ \dots \ v_{k-2}]^*\widehat S_R(\lambda)\big)=k-2
\end{equation}
for all $\lambda\in\mathbb C$. Without loss of generality we may assume that each
$v_i^*x_j(\lambda)$ has degree $m_j$ as a polynomial in $\lambda$. Note that this is
a generic condition, so if, by chance, this is not satisfied, then a small perturbation
of the vectors $v_1,\dots,v_{k-2}$ will ensure that the condition is satisfied. If the
perturbation is small enough then this does not change the rank condition~\eqref{eq:11.6.24}. Next, select $v_{k-1}\in\mathbb C^n$
such that $v_{k-1}^*x_j(\lambda)$ has $m_j$ zeros for $j=1,\dots,k$. (Note that
the latter condition is again generic, so such an $v_{k-1}$ exists.) Then by an
argument similar to the one in the case $k=2$ there can be only finitely many distinct
values $\lambda_1, \dotsc,\lambda_r$ that are solutions of
\begin{equation}\label{eq:inductionk}
\rank\big([v_1\ \dots\ v_{k-1}]^*S_R(\lambda)\big) = k-2.
\end{equation}
Indeed, we have that the rank in~\eqref{eq:inductionk} is less then $k-1$ if all
$(k-1)\times (k-1)$ minors are zero and there
is only a finite number of common zeros of these polynomials. Also note
that the rank in~\eqref{eq:inductionk} cannot be smaller than $k-2$,
because of~\eqref{eq:11.6.24}. In a same way as for $k=2$, if $r=0$, we are done.
Otherwise, since the vectors $x_1(\lambda_1), \dotsc, x_k(\lambda_1)$ are linearly
independent, there exists vector $w_{k-1}$ such that
$w_{k-1}^*x_j(\lambda_1)=0$ for $j=1, \dotsc,k-\ch{1}$ and
$w_{k-1}^*x_{\ch{k}}(\lambda_1)\ne 0$.
For a sufficiently small $\delta_{k-1}>0$ we then have
\begin{align}
[v_1\ \dots\ v_{k-2}\ \ v_{k-1} + \delta_{k-1} w_{k-1}]^*S_R(\lambda_1)\nonumber
& = [v_1\ \dots\ v_{k-1}]^*S_R(\lambda_1)\\
& \phantom{MMM} + \delta_{k-1} \,w_{k-1}^*x_{k-1}(\lambda_1)\,e_{k-1} \,e_{k}^T\label{eq:rank_is_full}
\end{align}
and the rank of \eqref{eq:rank_is_full} is $k-1$. If we
replace $v_{k-1}$ by $v_{k-1} + \delta_{k-1} w_{k-1}$
then $\lambda_1$ is not a solution of \eqref{eq:inductionk} anymore
and due to the continuity of zeros, there can now be at most $r-1$ values $\lambda$ such
that~\eqref{eq:inductionk} holds. We can remove them in the same way one by one by
further perturbing $v_{k-1}$. This delivers the desired example
for which \eqref{eq:rank_def} holds.\qed
\end{proof}

\begin{athm}\label{cnj:need_to_show}
 Let $P(\lambda)$ be an $n\times n$ singular matrix polynomial of degree
$d$ and normal rank $n-k$, and
let $S_R(\lambda)=\left[x_1(\lambda)\ \dots\ x_{k}(\lambda)\right]$
be an $n\times k$ matrix polynomial, such that its columns form a minimal basis
of the right nullspace of
$P(\lambda)$
with right minimal indices $m_1, \dotsc,m_k$.
Then, there exists $V\in\mathbb C^{n\times k}$ such that the polynomial
$\det\big(V^*S_R(\lambda)\big)$ has exactly $M=m_1+\cdots+m_k$ \emph{simple} zeros.
\end{athm}

\begin{proof}
For $k=1$ we can take $V=v\in\CC^n$ and let
$g(\lambda):=V^*S_R(\lambda)=\sum_{j=1}^{m_1}{v_j^*x_{1j}(\lambda)}$. Since
${\rm gcd}\big(x_{11}(\lambda), \dotsc, x_{1m_1}(\lambda)\big)=1$, it follows from
Lemma \ref{lem:newk} that generically the polynomial $g(\lambda)$ has exactly $m_1$ simple zeros.

For $k\ge 2$, select $v_1, \dotsc,v_{k-1}\in\CC^n$ satisfying~\eqref{eq:rank_def}
according to Lemma~\ref{cnj:k3} and an additional $v_k\in\mathbb C^n$, and set
$G_1(\lambda)=[v_1\ \dots\ v_k]^*S_R(\lambda)$. Without loss of generality
we may assume that $g_1(\lambda)=\det(G_1(\lambda))$
has degree $M=m_1+\cdots+m_k$. Indeed, since this is a generic condition, we can
slightly perturb $v_1,\dots,v_{k}$ to meet the condition, in the case that, by chance,
this is not satisfied. If the perturbation is sufficiently small, then this will not
change the rank condition~\eqref{eq:rank_def}. It follows that $g_1$ has exactly
$M$ roots $\xi_1, \dotsc,\xi_M$. At this point, we do not know if the roots are
pairwise distinct. Therefore, we next consider
$G_2(\lambda)=[v_1\ \dots\ v_{k-1}\ \ w]^*S_R(\lambda)$,
where the vectors $v_1, \dotsc,v_{k-1}$ are as before and $w\in\CC^n$. Again, we can
assume without loss of generality that $g_2(\lambda)=\det G_2(\lambda)$ has degree $M$.
By construction, the first $k-1$ rows of $G_2(\lambda)$ have full rank for all $\lambda$.
Consequently, a fixed value $\xi$ will generically not be among
the roots of $g_2(\lambda)=\det(G_2(\lambda))$ as a polynomial in $\lambda$ for
$i=1, \dotsc,M$. Indeed, since $S_R(\xi)$ has full rank $k$, the range of
$S_R(\xi)^*$ is $k$-dimensional and $S_R(\xi)^*v_1,\dots,S_R(\xi)^*v_{k-1}$
are linearly independent. Choosing $w$ such that $S_R(\xi)^*w$ completes these
vectors to a basis of the range of $S_R(\xi)^*$ proves that $G_2(\xi)$ is
nonsingular. But then the set of all $w$ for which $g_2(\xi)\neq 0$ is a
generic set. Since the intersection of finitely many generic sets is still generic, it follows that we can generically exclude all values $\xi_1, \dotsc,\xi_M$ from the
zeros $\eta_1, \dotsc,\eta_M$ of $g_2(\lambda)$.

It follows that ${\rm gcd}\big(g_1(\lambda),g_2(\lambda)\big)=1$ and we can apply Lemma
\ref{lem:newk}. Therefore, there exists $(s_1,s_2)\in\CC^2$ such that
$s_1g_1(\lambda)+ s_2 g_2(\lambda)$ has $M$ simple zeros. If we take
$V=[v_1\ \dots\ v_{k-1}\ \ \overline{s_1}v_k+\overline{s_2}w]$, then
$\det\big(V^*R_S(\lambda)\big)$ has $M$ simple roots which completes the proof.\qed
\end{proof}

\end{document}